\documentclass[english]{article}
\usepackage[T1]{fontenc}
\usepackage{amssymb,amsthm,amsmath,yhmath}
\usepackage[affil-it]{authblk}
\usepackage{enumitem}

\usepackage[pdftex=true,colorlinks=true]{hyperref}
\usepackage[all]{xy}
\SelectTips{cm}{}

\newcommand{\bbN}{{\mathbb N}}
\newcommand{\bbQ}{{\mathbb Q}}

\newcommand{\bbR}{{\mathbb R}}
\newcommand{\bbZ}{{\mathbb Z}}
\newcommand{\bbC}{{\mathbb C}}

\newcommand{\bbF}{\mathbb{F}}

\newcommand{\bfG}{{\mathbf G}}

\newcommand{\bR}{{\bf R}}

\newcommand{\bfH}{{\bf H}}

\newcommand{\id}{\operatorname{id}}

\newcommand{\supp}{\operatorname{supp}}

\newcommand{\Prob}{\operatorname{Prob}}

\newcommand{\End}{\operatorname{End}}
\newcommand{\Homeo}{\operatorname{Homeo}}
\newcommand{\Stab}{\operatorname{Stab}}
\newcommand{\SL}{\operatorname{SL}}
\newcommand{\Fix}{\operatorname{Fix}}

\newcommand{\Sym}{\operatorname{Sym}}

\newcommand{\PGL}{\operatorname{PGL}}
\newcommand{\GL}{\operatorname{GL}}
\newcommand{\Gr}{\operatorname{Gr}}
\newcommand{\M}{\operatorname{M}}

\newcommand{\LL}{\operatorname{L}}
\newcommand{\action}{\curvearrowright}

\newcommand{\lppar}{(\!(}
\newcommand{\rppar}{)\!)}

\DeclareMathOperator{\dd}{d\!}

\newtheorem{theorem}{Theorem}[section]
\newtheorem{lemma}[theorem]{Lemma}

\newtheorem{corollary}[theorem]{Corollary}
\newtheorem{cor}[theorem]{Corollary}
\newtheorem{proposition}[theorem]{Proposition}
\newtheorem{prop}[theorem]{Proposition}

\theoremstyle{definition}
\newtheorem{definition}[theorem]{Definition}
\newtheorem{defn}[theorem]{Definition}
\newtheorem{convention}[theorem]{Convention}

\newtheorem{example}[theorem]{Example}

\newtheorem{remark}[theorem]{Remark}

\newtheorem{setup}[theorem]{Setup}

\numberwithin{equation}{section}
\begin{document}
\title{Almost algebraic actions of algebraic groups and applications to algebraic representations}
\author{Uri Bader\thanks{uri.bader@gmail.com}}
\affil{Uri Bader, Weizmann Insitute of Science, Rehovot, Israel.}
\author{Bruno Duchesne\thanks{bruno.duchesne@univ-lorraine.fr}}
\affil{Universit\'e de Lorraine\\Institut \'Elie Cartan\\ B.P. 70239\\54506 Vandoeuvre-l\`es-Nancy Cedex\\ France.}
\author{Jean L\'ecureux\thanks{jean.lecureux@math.u-psud.fr}}
\affil{D\'epartement de Math\'ematiques\\ B\^atiment 425\\ Facult\'e des Sciences d'Orsay\\ Universit\'e Paris-Sud 11\\91405 Orsay\\ France.}
\maketitle

\begin{abstract}
Let $G$ be an algebraic group over a complete separable valued field $k$. 
We discuss the dynamics of the $G$-action on spaces of probability measures on algebraic $G$-varieties.
We show that the stabilizers of measures are almost algebraic and the orbits are separated by open invariant sets. 
We discuss various applications, including existence results for algebraic representations of amenable ergodic actions.
The latter provides an essential technical step in the recent generalization of Margulis-Zimmer super-rigidity phenomenon \cite{BF}.
\end{abstract}

%

\section{Introduction}

This work concerns mainly the dynamics of an algebraic group acting on the space of probability measures on an algebraic variety.
Most (but not all) of our results are known for local fields (most times, under a characteristic zero assumption). 
Our main contribution is giving an approach which is applicable also to a more general class of fields: complete valued fields.
On our source of motivation, which stems from ergodic theory, we will elaborate in  \S\ref{applications}, and in particular Theorem~\ref{BDL}.
First we describe our objects of consideration and our main results, put in some historical context.

\begin{setup} \label{setup}
For the entire paper $(k,|\cdot|)$ will be a valued field, which is assumed to be complete and separable as a metric space, and $\widehat{k}$ will be the completion of its algebraic closure, endowed with the extended absolute value.
\end{setup}

Note that $\widehat{k}$ is separable and complete as well (see the proof of Proposition~\ref{polishing}).
The most familiar examples of separable complete valued fields are of course $\bbR$ and $\bbC$, but one may also consider the $p$-adic fields $\bbQ_p$, as well as their finite extensions.
Considering $k=\bbC_p=\widehat{\bbQ}_p$ one may work over a field which is simultaneously complete, separable and algebraically closed. 
Another example of a complete valued field is given by fields of Laurent series $K\lppar t\rppar$, where $K$ is any field (this field is local if and only if $K$ is finite, and separable if and only if $K$ is countable),
or more generally the field of Hahn series $K\lppar t^{\Gamma}\rppar$, where $\Gamma$ is a subgroup of $\bR$ (see for example \cite{MR1225257}). This field is separable if and only if $K$ is countable and $\Gamma$ is discrete (see \cite{MOquestion}).




\begin{convention}Algebraic varieties over $k$ will be identified with their $\widehat{k}$-points and will be denoted by boldface letters. Their $k$-points will be denoted by corresponding Roman letters. In particular we use the following.
\end{convention}

\begin{setup}\label{setupG}
We fix  a $k$-algebraic group ${\bf G}$ and we denote $G={\bf G}(k)$.
\end{setup}

We are interested in algebraic dynamical systems, which we now briefly describe.
For a formal, pedantic description see \S\ref{alg perlim} and in particular Proposition~\ref{polishing}.
By an algebraic dynamical system we mean the action of $G$ on $V$, where
$V$ is the space of $k$-points of a $k$-algebraic variety ${\bf V}$ on which ${\bf G}$ acts $k$-morphically.
Such a dynamical system is Polish: $G$ is a Polish group, $V$ a Polish space and the action map $G\times V\to V$ is continuous (see \S\ref{alg perlim} for proper definitions).
The point stabilizers of such an action are algebraic subgroups, and by a result of Bernstein-Zelevinski~\cite{b-z}, the orbits of such an action are locally closed
(see Proposition \ref{polishing}).


Following previous works of Furstenberg and Moore, Zimmer found a surprising result: 
for the action of an algebraic group $G$ on an algebraic variety $V$, all defined over $\bbR$, 
consider now the action of $G$ on the space $\Prob(V)$ of probability measures on $V$. Then  the point stabilizers are again algebraic subgroups and the orbits are locally closed.
However, this result does not extend trivially to other fields.
For example, with $k=\bbC$, consider the Haar measure on the circle $S^1<\bbC^*$.
For the action of $\bbC^*$ on itself, the stabilizer of that measure is $S^1$, which is not a $\bbC$-algebraic subgroup.
Similarly, for $k=\bbQ_p$, consider the Haar measure on the $p$-adic integers $\bbZ_p<\bbQ_p$.
For the action of $\bbQ_p$ on itself, the stabilizer of that measure is $\bbZ_p$, which is not a $\bbQ_p$-algebraic subgroup.

\begin{definition} \label{aag-aaa}
A closed subgroup $L<G$ is called \emph{almost algebraic} if there exists a $k$-algebraic subgroup ${\bf H}<{\bf G}$
such that $L$ contains $H={\bf H}(k)$ as a normal cocompact subgroup.
A continuous action of $G$ on a Polish space $V$ is called \emph{almost algebraic} if the point stabilizers are almost algebraic subgroups of $G$ and the collection of $G$-invariant open sets separates the $G$-orbits, i.e the quotient topology on $G\backslash V$ is $T_0$.
\end{definition}

\begin{remark}
If $k$ is a local field then $G$ is locally compact
and by \cite[Theorem~2.6]{effros} the condition $G\backslash V$ is $T_0$
is equivalent to the (a priori stronger) condition that every $G$-orbit is locally closed in $V$.
\end{remark}

\begin{remark}
If $k=\bbR$ then every compact subgroup of $G$ is the real points of a real algebraic subgroup of ${\bf G}$ (see e.g. \cite[Chapter 4, Theorem 2.1]{MR1056486}).
It follows that every almost algebraic subgroup is the real points of a real algebraic subgroup of ${\bf G}$.
We get that a continuous action of $G$ on a Polish space $V$ is almost algebraic if and only if the stabilizers are real algebraic and the orbits are locally closed.
\end{remark}

Two obvious classes of examples of almost algebraic actions are algebraic actions (by the previously mentioned result of Bernstein-Zelevinski)
and proper actions (as the stabilizers are compact and the space of orbits is $T_2$, that is, Hausdorff).
The notion of almost algebraic action is a natural common generalization.
It is an easy corollary of Prokhorov's theorem (see Theorem~\ref{Prok} below) that if the action of $G$ on $V$ is proper then so is its action on $\Prob(V)$, see Lemma~\ref{Prokhorov}.
The main theorem of this paper is the following analogue.

\begin{theorem} \label{mainthm}
If the action of $G$ on a Polish space $V$ is almost algebraic then the action of $G$ on $\Prob(V)$ is almost algebraic as well.
\end{theorem}

The following corollary was obtained by Zimmer, under the assumptions that $k$ is a local field of characteristic 0
and ${\bf V}$ is homogeneous, see \cite[Chapter~3]{zimmer-book}.

\begin{cor}
Assume ${\bf G}$ has a $k$-action on a $k$-variety ${\bf V}$.
Then the induced action of $G={\bf G}(k)$ on $\Prob({\bf V}(k))$ is almost algebraic.
\end{cor}

In the course of the proof of Theorem~\ref{mainthm} we obtain in fact a more precise information. A $k$-${\bf G}$-variety is a $k$-variety with a $k$-action of $\bf G$.

\begin{prop} \label{prop:stab}
Fix a closed subgroup $L<G$.
Then there exists a $k$-subgroup ${\bf H}_0<{\bf G}$ which is normalized by $L$
such that $L$ has a precompact image in the Polish group $(N_{\bf G}({\bf H}_0)/{\bf H}_0)(k)$
and such that for every $k$-${\bf G}$-variety ${\bf V}$, any $L$-invariant finite measure on ${\bf V}(k)$
is supported on the subvariety of ${\bf H}_0$-fixed points, ${\bf V}^{{\bf H}_0}\cap {\bf V}(k)$.
\end{prop}

This proposition is a generalization of one of the main results of Shalom~\cite{Shalom}, who proves it under the assumptions that $k$ is local and $L=G$.
For the case $L=G$ the following striking corollary is obtained.

\begin{cor}
If for every strict $k$-algebraic normal subgroup ${\bf H}\triangleleft {\bf G}$, ${\bf G}(k)/{\bf H}(k)$ is non-compact, then every $G$-invariant measure on any $k$-${\bf G}$-algebraic variety ${\bf V}(k)$ is supported on the ${\bf G}$-fixed points.
\end{cor}

In particular we can deduce easily the Borel density theorem.

\begin{corollary} 
Let ${\bf G}$ be a $k$-algebraic group and $\Gamma<G={\bf G}(k)$ be a closed subgroup such that $G/\Gamma$ has a $G$-invariant probability measure. If for every proper $k$-algebraic normal subgroup ${\bf H}\triangleleft {\bf G}$, ${\bf G}(k)/{\bf H}(k)$ is non-compact, then $\Gamma$ is Zariski dense in ${\bf G}$.
\end{corollary}

To deduce the last corollary from the previous one, consider the map 
\[ G/\Gamma\to ({\bf G}/\overline{\Gamma}^Z)(k), \] 
where $\overline{\Gamma}^Z$ denotes the Zariski closure of $\Gamma$, and push forward the invariant measure from $G/\Gamma$ to obtain a $G$-invariant measure on $({\bf G}/\overline{\Gamma}^Z)(k)$. The homogeneous space ${\bf G}/\overline{\Gamma}^Z$ must contain a ${\bf G}$-fixed point, hence must be trivial. That is $\overline{\Gamma}^Z={\bf G}$.

\subsection{Applications: ergodic measures on algebraic varieties}

A classical theme in ergodic theory is the attempt to classify all ergodic measures classes,
given a continuous action of a topological group on a Polish space.
In this regard, the axiom that the space of orbits is $T_0$ has strong applications.
 Recall that, given a group $L$ acting by homeomorphisms on a Polish space $V$, a measure on $V$ is \emph{$L$-quasi-invariant} if its class is $L$-invariant. The following proposition is well known.

\begin{prop} \label{tame}
Let $V$ be a Polish $G$-space and assume that the quotient topology on $G\backslash V$ is $T_0$.
Let $L<G$ be a subgroup and $\mu$ an $L$-quasi-invariant ergodic probability (or $\sigma$-finite) measure.
Then there exists $v\in V$ such that $\mu(V-Gv)=0$.
\end{prop}

Indeed, $G\backslash V$ is second countable, as $V$ is, and for a countable basis $B_i$, denoting the push forward of $\mu$ to $G\backslash V$ by $\bar{\mu}$,
the set 
\[ \bigcap \{B_i ~|~\bar{\mu}(B_i)=1\}\cap \bigcap\{B_i^c ~|~\bar{\mu}(B_i)=0\} \]
is clearly a singleton, whose preimage in $V$ is an orbit of full measure.

In particular, we get that for a subgroup $L<G$ and an algebraic dynamical system of $G$,
every $L$-invariant measure is supported on a single $G$-orbit.
Another striking result is that an algebraic variety cannot support a weakly mixing probability measure.
Recall that an $L$-invariant probability measure $\mu$ is weakly mixing if and only if $\mu\times\mu$ is $L$-ergodic.

\begin{cor}
Assume ${\bf G}$ has a $k$-action on the $k$-variety ${\bf V}$.
Fix a closed subgroup $L<G$ and let $\mu$ be an $L$-invariant weakly mixing probability measure on $V={\bf V}(k)$.
Then there exists a point $x\in V^L$ such that $\mu=\delta_x$.
\end{cor}

This corollary follows at once from Proposition~\ref{prop:stab}, as the action of $L$ on ${\bf V}^{{\bf H}_0}\cap {\bf V}(k)$ is via a compact group.


We end this subsection with the following useful application, obtained by composing Proposition~\ref{tame} with Theorem~\ref{mainthm}.
This corollary is in fact our main motivation for developing the material in this paper.
It deals with measure on spaces of measures, and is
the main tool in deriving Theorem~\ref{BDL} below.

\begin{cor}
Assume ${\bf G}$ has a $k$-action on the $k$-variety ${\bf V}$. Denote $V={\bf V}(k)$.
Let $L<G$ be a subgroup and $\nu$ be an $L$-ergodic quasi-invariant measure on $\Prob(V)$.
Then there exists $\mu\in \Prob(V)$ such that $\nu(\Prob(V)-G\mu)=0$.
\end{cor}

\subsection{Applications to algebraic representations of ergodic actions.} \label{applications}

A main motivation for us to extend the foundation outside the traditional local field zone is the recent developments in the theory of algebraic representations of ergodic actions, and in particular its applications to rigidity theory.
In \cite{BF} the following theorem, as well as various generalizations, are proven.

\begin{theorem}[{\cite[Theorem 1.1]{BF}}, Margulis super-rigidity for arbitrary fields] \label{marguliscor}
Let $l$ be a local field.
Let $T$ to be the $l$-points of a connected almost-simple algebraic group defined over $l$.
Assume that the $l$-rank of $T$ is at least two.
Let $\Gamma<T$ be a lattice.

Let $k$ be a valued field.
Assume that as a metric space $k$ is complete.
Let $G$ be the $k$-points of an adjoint simple algebraic group defined over $k$.
Let $\delta:\Gamma \to G$ be a homomorphism.
Assume $\delta(\Gamma)$ is Zariski dense in $G$ and unbounded.
Then there exists a continuous homomorphism $d:T\to G$
such that $\delta=d|_{\Gamma}$.
\end{theorem}

The proofs in \cite{BF} are based on the following, slightly technical, theorem which will be proven here.

\begin{theorem} \label{BDL}
Let $R$ be a locally compact group and $Y$ be an ergodic, amenable Lebesgue $R$-space.
Let $(k,|\cdot|)$ be a valued field.
Assume that as a metric space $k$ is complete and separable.
Let ${\bf G}$ be a simple $k$-algebraic group.
Let $f:R\times Y \to {\bf G}(k)$ be a measurable cocycle.

Then either
there exists
a $k$-algebraic subgroup ${\bf H}\lneq {\bf G}$
and an $f$-equivariant measurable map $\phi:Y\to {\bf G}/{\bf H}(k)$,
or there exists
a complete and separable metric space $V$ on which $G$ acts by isometries
with bounded stabilizers
and an $f$-equivariant measurable map $\phi':Y\to V$.
%
\end{theorem}

A more friendly, cocycle free, version is the following.

\begin{cor}\label{thm:rep}
Let $R$ be a locally compact, second countable group. Let $Y$ be an ergodic, amenable $R$-space. 
Suppose that $\bfG$ is an adjoint  simple $k$-algebraic group, and there is a morphism $R\to G=\bfG(k)$.
Then :
\begin{itemize} 
\item Either there exists a complete and separable metric space $V$, on which $G$ acts by isometries with bounded stabilizers,  and an $R$-equivariant measurable map $Y\to V$; 
\item or there exists a strict $k$-algebraic subgroup $\bfH$ and an $R$-equivariant measurable map $Y\to\bfG/\bfH(k)$.
\end{itemize}
\end{cor}

Taking $Y$ to be a point in the above corollary, we obtain the following.

\begin{cor}
Suppose $R<\GL_n(k)$ is a closed amenable subgroup.
Then the image of $R$ in $\overline{R}^Z$ modulo its solvable radical is bounded.
\end{cor}

Indeed, upon moding out the solvable radical of $\overline{R}^Z$, the latter is a product of simple adjoint factors, 
and by the previous corollary the image of $R$ in each factor is bounded.

Note that over various fields, such as $\bbC_p$ and $\bar{\bbF}_p\lppar t\rppar$, every bounded group is amenable, being the closure of an ascending union of compact groups, while for other fields there exist bounded groups which are not amenable.
For example $\SL_2(\bbQ[[t]])$, which is bounded in $\SL_2(\bbQ\lppar t\rppar)$, factors over the discrete group $\SL_2(\bbQ)$ which contains a free group.

\subsection{The structure of the paper.}

The paper has two halves: the first half consisting of \S2,\S3 is devoted to the proof of Theorem~\ref{mainthm}
and the second half is devoted to the proof of Theorem~\ref{BDL}.

In \S2 we collect various needed preliminaries, in particular we discuss the Polish structure on algebraic varieties, 
and on spaces of measures. The most important results in this section are Proposition~\ref{polishing} that discusses 
algebraic varieties and 
and Corollary~\ref{ec-replacement} that uses disintegration as a replacement for a classical ergodic decomposition argument
(which is not applicable in our context, due to the lack of compactness).
The heart of the paper is \S3, where the concept of almost algebraic action is discussed.
Theorem~\ref{mainthm} is proven at \S\ref{proofof}.

In \S4, we give a thorough discussion of bounded subgroups of algebraic groups,
and in \S5, we discuss a suitable replacement of a compactification of coset spaces. 
In \S6, we prove Theorem~\ref{BDL}.

\subsection*{Acknowledgements}

U.B was supported in part by the ERC grant 306706. B.D. was supported in part by Lorraine Region and Lorraine University. B.D \& J.L. were supported in part by ANR grant ANR-14-CE25-0004 GAMME.

\section{Preliminaries}

\subsection{Algebraic varieties as Polish spaces} \label{alg perlim}

Recall that a topological space is called Polish if it is separable and completely metrizable.
For a good survey on the subject we recommend \cite{kechris}.
We mention that the class of Polish spaces is closed under countable disjoint unions and countable products.
A $G_\delta$-subset of a Polish space is Polish so, in particular, a locally closed subset of a Polish space is Polish.
A Hausdorff space which admits a finite open covering by Polish open sets is itself Polish.
Indeed,
such a space is clearly metrizable (e.g.\ by Urysohn metrization theorem \cite[Theorem 1.1]{kechris})
so it is Polish by
Sierpinski theorem \cite[Theorem 8.19]{kechris} which states that the image of a continuous open map from a Polish space to a separable metrizable space is Polish.

A topological group which underlying topological space is Polish is called a Polish group.
Sierpinski theorem also implies that for a Polish group $K$ and a closed subgroup $L$, the quotient topology on $K/L$ is Polish.
Effros Lemma \cite[Lemma 2.5]{effros} says that the quotient topology on $K/L$ is the unique $K$-invariant Polish topology on this space. Another important result of Effros concerning Polish actions (that are continuous actions of Polish groups on Polish spaces) is the following.

\begin{theorem}[Effros theorem {\cite[Theorem 2.1]{effros}}] \label{effros}
For a continuous action of a Polish group $G$ on a Polish space $V$ the following are equivalent.
\begin{enumerate}
\item The quotient topology on $G\backslash V$ is $T_0$.
\item For every $v\in V$, the orbit map $G/\Stab_G(v) \to Gv$ is a homeomorphism.
\end{enumerate}
\end{theorem}

Our basic class of Polish actions will be given by actions of algebraic groups on algebraic varieties. As mentioned in Setups~\ref{setup} \& \ref{setupG}, we fixed a complete and separable valued field $(k,|\cdot|)$, that is a field $k$ with an absolute value $|\cdot|$
which is complete and separable (in the sense of having a countable dense subset). See \cite{valued,MR746961}\footnote{In the second reference, the word valuation is used for what we call an absolute value.} for a general discussion on these fields. It is a standard fact that a complete absolute value on a field $F$ has a unique extension to its algebraic closure $\overline{F}$ \cite[\S3.2.4, Theorem 2]{MR746961} and Hensel lemma implies that the completion $\widehat{F}$ of this algebraic closure is still algebraically closed \cite[\S3.4.1, Proposition 3]{MR746961}.

Recall that we identify each $k$-variety $\bf V$ with its set of $\widehat{k}$-points. In particular, this identification yields a topology on $\bf V$. Identifying the affine space $\mathbb{A}^n(\widehat{k})$ with $\widehat{k}^n$, any affine $k$-variety can be seen as a closed subset of $\mathbb{A}^n(\widehat{k})$. More generally, a $k$-variety has a unique topology making its affine charts homeomorphisms. 
Observe that with this topology, the set of $k$-points $V$ of $\bf V$ is closed.

Topological notions, unless otherwise said, will always refer to this topology.
In particular, for the $k$-algebraic group $\bf G$ we fixed,  ${\bf G}$ and $G=\mathbf{G}(k)$ are topological groups.
We note that $V$ actually carries a structure of a $k$-analytic manifold, $G$ is a $k$-analytic group and
the action of $G$ on $V$ is $k$-analytic.
We will not make an explicit use of the analytic structure here.
The interested reader is referred to the excellent text \cite{serre}, in which the theory of analytic manifolds and Lie groups over complete valued fields is developed (see in particular \cite[Part II, Chapter I]{serre}). 

We will discuss the category of $k$-${\bf G}$-varieties.
A \emph{$k$-${\bf G}$-variety} is a $k$-variety endowed with an algebraic action of ${\bf G}$ which is defined over $k$.
A morphism of such varieties is a $k$-morphism which commutes with the ${\bf G}$-action.

\begin{proposition} \label{polishing}
A $k$-variety $\bf V$ and its set of $k$-points $V$ are Polish spaces. In particular, $\bf{G}$ and $G$ are Polish groups.

If ${\bf V}$ is a $k$-${\bf G}$-variety then the $G$-orbits in $V$ are locally closed and the quotient topology on $G\backslash V$ is $T_0$. For $v\in V$, the orbit ${\bf G}v$ is a $k$-subvariety of ${\bf V}$. There exists a $k$-subgroup ${\bf H}<{\bf G}$ contained in the stabilizer of $v$ such that the orbit map ${\bf G}/{\bf H}\to {\bf G}v$ is defined over $k$ and the induced map $G/H \to Gv$ is a homeomorphism, where $H={\bf H}(k)$, $G/H$ is endowed with the quotient space topology and $Gv$ is endowed with the subspace topology.
\end{proposition}

\begin{proof}Let us first explain how the extended absolute value  makes $\widehat{k}$ Polish. 
In our situation $k$ has a countable dense subfield $k_0$. The algebraic closure $\overline{k_0}$ of $k_0$ is still countable and thus its completion $\widehat{k_0}$ is separable  and algebraically closed.  By the universal property of the algebraic closure, $\overline{k}$ embeds in  $\widehat{k_0}$ and by uniqueness of the extension of the absolute value, this embedding is an isometry. Thus $\widehat{k}$ is algebraically closed, complete and separable.

Since $\widehat{k}$ is Polish, so is the affine space $\mathbb{A}^n(\widehat{k}) \simeq \widehat{k}^n$. It follows that $\bf V$ (respectively $V$) is a Polish space, as this space is a Hausdorff space which admits a finite open covering by Polish open sets --- the domains of its $k$-affine charts (respectively their $k$-points).

The fact that the $G$-orbits in $V$ are locally closed is proven in the appendix of \cite{b-z}.
Note that in \cite{b-z} the statement is claimed only for non-Archimedean local fields, but the proof is actually correct for any field with
 complete non-trivial absolute value, which is the setting of \cite[Part II, Chapter III]{serre}
on which \cite{b-z} relies. Another proof can be found in \cite[\S0.5]{French}. It is then immediate that the quotient topology on $G\backslash V$ is $T_0$.


For $v\in V$ the orbit ${\bf G}v$ is a $k$-subvariety of ${\bf V}$ by \cite[Proposition 6.7]{borel}.
We set ${\bf H}=\overline{\Stab_G(v)}^Z$
(note that if $\operatorname{char}(k)=0$ then ${\bf H}=\Stab_{\bf G}(v)$).
By \cite[AG, Theorem~14.4]{borel}, ${\bf H}$ is defined over $k$, and it is straightforward
that $H={\bf H}(k)=\Stab_G(v)$.
By \cite[Theorem 6.8]{borel} the orbit map ${\bf G}/{\bf H} \to {\bf G}v$ is defined over $k$,
thus it restricts to a continuous map from $G/H$ onto $Gv$.
The fact that the latter map is a homeomorphism follows from Effros theorem (Theorem~\ref{effros}) since $G\backslash V$ is $T_0$.
\end{proof}

We emphasize that, as a special case of Proposition \ref{polishing}, we get that for every $k$-algebraic subgroup $\bf H$ of $\bf G$, the embedding $G/H\to {\bf G}/{\bf H}(k)$ is a homeomorphism on its image. We will use this fact freely in the sequel.

\subsection{Spaces of measures as Polish spaces}

In this subsection $V$ denotes a Polish space. 
We let $\Prob(V)$ be the set of Borel probability measures on $V$, endowed with the weak*-topology 
(also called \emph{the topology of weak convergence}). 
This topology comes from the embedding of $\Prob(V)$ in the dual of the Banach space of bounded continuous functions on $V$. 
If $d$ is a  complete metric on $V$ which is compatible with the topology (the metric topology coincides with the original topology on $V$), the corresponding Prokhorov metric $\bf{d}$ on $\Prob(V)$ is defined as follows:
for $\mu,\nu\in\Prob(V)$, $\bf{d}(\mu,\nu)$ is the infimum of $\varepsilon>0$ such that for all Borel subset $A\subseteq V$, $\mu(A)\leq\nu(A_\varepsilon)+\varepsilon$ and symmetrically $\nu(A)\leq\mu(A_\varepsilon)+\varepsilon$, where $A_\varepsilon$ is the $\varepsilon$-neighborhood (for $d$) around $A$.
The following theorem summarizes some standard results, see Chapter 6 and Appendix III of \cite{Billingsley}.

\begin{theorem}[Prokhorov]\label{Prok}
The metric space $(\Prob(V),\bf{d})$ is complete and separable and the topology induced by $\bf{d}$ on $\Prob(V)$ is the weak*-topology.
In particular the space $\Prob(V)$ endowed with the weak*-topology is Polish.

A subset $C$ in $\Prob(V)$ is precompact if and only if it is tight: for every $\epsilon>0$ there exists compact $K\subset V$ such that for every $\mu\in C$,
$\mu(K)>1-\epsilon$.
In particular $\Prob(V)$ is compact if $V$ is.
\end{theorem}

\begin{remark}
Replacing if necessary $d$ by a bounded metric, we note that there is another metric on $\Prob(V)$ with the same properties (metrizing the weak*-topology and being invariant under isometries): the Wasserstein metric \cite[Corollary 6.13]{Villani}. 
\end{remark}

We endow $\Homeo(V)$ with the pointwise convergence topology. 
The following is a standard application of the Baire category theorem, see {\cite[Theorem 9.14]{kechris}}.

\begin{theorem}\label{cont-action}
Assume $G$ is acting by homeomorphisms on $V$. Then the action map $G\times V \to V$ is continuous if and only if the homomorphism $G\to\Homeo(V)$ is continuous.
\end{theorem}

\begin{lemma}\label{contprob} 
If $G$ acts continuously on $V$ then it also acts continuously on $\Prob(V)$ and if the action $G\action (V,d)$ is by isometries, the action $G\action(\Prob(V),{\bf d})$ is also by isometries.
 \end{lemma} 

\begin{proof} 
The fact that $G$ acts by isometries on $\Prob(V)$ when $G$ acts by isometries on $V$ is straightforward from the definition of the Prokhorov metric.
In order to prove that $G$ acts continuously on $\Prob(V)$ when it acts continuously on $V$ it is enough, by Theorem~\ref{cont-action}, to show that for every $\mu\in \Prob(V)$
and every sequence $g_n$ in $G$, $g_n\to e$ in $G$ implies $g_n\mu\to \mu$ in $\Prob(V)$.
Fix $\mu\in \Prob(V)$ and assume $g_n \to e$ in $G$.
For every bounded continuous function $f$ on $V$, we have by Lebesgue bounded convergence theorem
\[ \int f(x) \dd \,(g_n\mu)(x) = \int f(g_nx)\dd\mu(x) \to \int f(x)\dd\mu(x) \]
as for every $x\in V$, $g_nx\to x$. Thus, by the definition of the weak*-topology $g_n\mu\to \mu$.
\end{proof}

We observe that Lemma \ref{contprob} and Proposition \ref{polishing} show that if $\bf V$ is a $k$-${\bf G}$-variety then $G$ acts continuously on $V={\bf V}(k)$ and on $\Prob(V)$.
The following is a nice application of Prokhorov theorem (Theorem~\ref{Prok}).

\begin{lemma} \label{Prokhorov}
If the action of $G$ on $V$ is proper then the action of $G$ on $\Prob(V)$ is proper as well.
\end{lemma}

\begin{proof}
For a compact $C\subset \Prob(V)$ we can find a compact $K\subset V$ with $\mu(K)>1/2$ for every $\mu\in C$ by Theorem~\ref{Prok}. 
Then for $g\in G$ and $\mu\in C$ such that $g\mu\in C$ we get that both $\mu(K)>1/2$ and $\mu(gK)=g\mu(K)>1/2$,
thus $gK\cap K\neq \emptyset$.
We conclude that $\{g\in G~|~gC\cap C\neq\emptyset\}$ is precompact, as it is a subset of the precompact set $\{g\in G~|~gK\cap K\neq\emptyset\}$.
\end{proof}

\subsection{Polish extensions and disintegration}



\begin{definition}
A \emph{Polish fibration} is a continuous map $p:V\to U$ where $U$ is a $T_0$-space and $V$ a Polish space. 
An action of $G$ on such a Polish fibration is a pair of continuous actions on $V$ and $U$ such that $p$ is equivariant.
\end{definition}

Let $p:V\to U$ be a Polish fibration. Let $\Prob_U(V)$ be the set of probability measures on $V$ which are supported on one fiber.  We denote $p_\bullet:\Prob_U(V)\to U$ the natural map.

\begin{lemma}\label{probuv} The map $p_\bullet$ is a Polish fibration. If the group $G$ acts on the Polish fibration $V\to U$, then it also acts on $p_\bullet$.
\end{lemma}

\begin{proof}
Since $U$ is $T_0$, fibers of $p$ are separated by  a countable family $(C_n)$  of closed saturated subsets of $V$. A probability measure $\mu$ is supported on one fiber if and only if for all $n$, $\mu(C_n)\mu(V\setminus C_n)=0$. The set $\{\mu\in\Prob(V),\ \mu(C_n)=1\}$ is closed and $\{\mu\in\Prob(V),\ \mu(V\setminus C_n)=1\}$ is $G_\delta$ since for all $0<r<1$, $\{\mu\in\Prob(V),\ \mu(V\setminus C_n)>r\}$ is open. So  $\Prob_U(V)$ is a $G_\delta$-subset of $\Prob(V)$ and thus Polish.

Let us show that $p_\bullet$ is continuous. Assume $\mu_n\to\mu$ in $\Prob_U(V)$. Let $u=p_\bullet(\mu)$ and $u_n=p_\bullet(\mu_n)$. Let $O\subseteq U$ be an open set containing $u$. For $n$ large enough, $\mu_n(p^{-1}(O))>1/2$ and thus $u_n\in O$. 

If $G$ acts on $V\to U$, it is clear that $G$ acts on $\Prob_U(V)$. The continuity of the action on $\Prob_U(V)$ follows from Lemma \ref{contprob}. 
\end{proof}

Let $(U,\nu)$ be a probability space and $X$ be a Polish space, we denote  by $\LL^0(U,X)$ the space of classes of measurable maps from $U$ to $X$, under the equivalence relation of equality $\nu$-almost everywhere.
Note that the dependence on $\nu$ is implicit in our notation.
We endow that space with the topology of convergence in probability. Fixing a compatible metric $d$ on $X$, this topology is metrized as follows: for  $\phi,\phi'\in\LL^0(U,X)$, the distance between $\phi$ and $\phi'$ is 
$$\delta(\phi,\phi')=\int_X\min(d(\phi(v),\phi'(v)),1)\,\dd \nu(v).$$ This topology can be also defined using sequences: $\phi_n\to\phi$ if for any $\varepsilon>0$, there is $A\subseteq U$ such that $\nu(A)>1-\varepsilon$ and for all $n$ sufficiently large and all $v\in A$, $d(\phi(v),\phi_n(v))<\varepsilon$. We note that this topology on $\LL^0(U,X)$ does not depend on the choice of an equivalent metric on $V$.  This turns $\LL^0(U,X)$ into a Polish space.

\begin{lemma} \label{convsubseq}
Assume $(\alpha_n)$ is a sequence converging to $\alpha$ in probability in $\LL^0(U,X)$.
Then there exists a subsequence $\alpha_{n_k}$ which convergence $\nu$-a.e. to $\alpha$,
that is for $\nu$-almost every $u\in U$, $\alpha_{n_k}(u)$ converges to $\alpha(u)$ in $X$.
\end{lemma}

The proof of the lemma is standard, but in most textbooks it appears only for the cases $X=\bbR$ or $X=\bbC$,
see for example \cite[Theorem~2.30]{Folland}.
Even though the standard proof works mutatis-mutandis, we give below a short argument, reducing the general case to the case $X=\bbR$.

\begin{proof}
Observe that the sequence $d(\alpha_n,\alpha)$ (which denotes the map $u\mapsto d(\alpha_n(u),\alpha(u))$) converges in probability to 0 in $\LL^0(U,\bbR)$.
Thus there exists a subsequence 
$d(\alpha_{n_k},\alpha)$ converging to 0 a.e, and we get that $\alpha_{n_k}$ converges to $\alpha$ a.e.
\end{proof}

If $p\colon V\to U$ is a Polish fibration, and $\nu$ is a measure on $U$, 
we denote $\LL^0_p(U,V)$ the space of measurable (identified if agree almost everywhere) sections of $p$, i.e. maps which associates to $u\in U$ a point in $p^{-1}(U)$, endowed with the induced topology from $\LL^0(U,V)$. If $G$ acts on the Polish fibration $p$, it also acts on $\LL^0_p(U,V)$ via the formula $(gf)(u)=gf(g^{-1}u)$ where $u\in U$ and $f\in\LL^0_p(U,V)$. 

The following theorem is a variation of the classical theorem of disintegration of measures. It is essentially proven in \cite{simmons}. 

\begin{theorem} \label{disintegration}
Let $p:V\to U$ be a Polish fibration and $\nu$ be a probability measure on $U$. Let $P=\{\mu\in\Prob(V)\mid p_*\mu=\nu\}$. 
For every $\alpha\in \LL^0_{p_\bullet}\left(U,\Prob_U(V)\right)$ the formula $\int_U\alpha(u) \dd\nu$ defines an element of $P$.
The map thus obtained $\LL^0_{p_\bullet}\left(U,\Prob_U(V)\right) \to P$ is a homeomorphism onto.
\end{theorem}

\begin{defn}
For $\mu\in P$, the element of $\LL^0_{p_\bullet}(U,\Prob_U(V))$ obtained by applying to $\mu$ the inverse map of $\alpha \mapsto \int_U\alpha(u) \dd\nu$ is denoted $u \mapsto \mu_u$. It is called the \textit{disintegration} of $\mu$ with respect to $p:V\to U$.
\end{defn}

\begin{proof} 
We first claim that the map $\alpha \mapsto \int_U\alpha(u) \dd\nu$ is continuous, and then
we argue to show that it is invertible, and its inverse is continuous as well.

For the continuity, given a converging sequence $\alpha_n\to \alpha$ in $\LL^0_{p_\bullet}(U,\Prob_U(V))$ with $\mu_n=\int_U\alpha_n(u) \dd\nu$, $\mu=\int_U\alpha(u) \dd\nu$,
it is enough to show that every subsequence of $\mu_n$ has a subsequence that converges to $\mu$.
Since every sequence that converges in measure has a subsequence that converges almost everywhere, abusing our notation 
and denoting again $\alpha_n$ and $\mu_n$ for the resulting sub-sub-sequences, we may assume that $\alpha_n$ converges to $\alpha$ $\nu$-almost everywhere.
Picking an arbitrary continuous bounded function $f$ on $V$, we obtain that for $\nu$-a.e $u\in U$, 
$\int_V \dd\alpha_n(u) f \to \int_V \dd\alpha(u) f$.
Thus by Lebesgue bounded convergence theorem we get
\[ \int_V \dd\mu_n f=\int_U \dd\nu \int_V \dd\alpha_n(u) f \to \int_U \dd\nu \int_V \dd\alpha(u) f = \int_V \dd\mu f. \]
This shows that indeed $\mu_n\to \mu$.

We now argue that the map $\alpha \mapsto \int_U\alpha(u) \dd\nu$ is invertible and its inverse is continuous. Without loss of generality, we can assume that $p$ is onto. Hence $U$ is second countable. Since it is also $T_0$, it follows that $U$ is countably separated.
By  \cite[Proposition~A.1]{zimmer-book}, there exists a Borel embedding $\phi:U \to [0,1]$. We  consider $[0,1]$ with the measure $\phi_*\nu$. 
Precomposition by $\phi$ gives a homeomorphism $\LL^0_{(\phi \circ p)_*}\left([0,1],\Prob_{[0,1]}(V)\right)\to \LL^0_{p_\bullet}\left(U,\Prob_U(V)\right)$.
Thus, in what follows we may and do assume that $U \subset [0,1]$.\footnote{Since the embedding $U\to [0,1]$ is only Borel, when we assume $U \subset [0,1]$, the fibration $V\to U$ cannot be assumed to be Polish anymore. Since our argument does not depend on the topology of $U$, this does not matter here.}
Under this assumption
\cite[Theorem 2.1]{simmons} guarantees that the map $\LL^0_{p_\bullet}(U,\Prob_U(V)) \to P$ is invertible.
We denote the preimage of $\mu\in P$ by $u\mapsto \mu_u$. 
We are left to show that this association is continuous.
To this end we embed $V$ in a compact metric space $V'$ and extend $p$ by setting $p(v')=1$ for $v'\in V'-V$.
Then  \cite[Theorem 2.2]{simmons}  proves
  that  for almost every $u\in U$, $\mu_u$  is obtained as the weak*-limit of the normalized restrictions, denoted by $\mu_{u,\eta}$, of $\mu$ on $p^{-1}(u-\eta,u+\eta)$ as $\eta\to0$.

Assume that $\mu^n\to\mu$ is a converging sequence in $P$.  We know that for $\nu$-a.e. $u$, ${\bf d}(\mu_{u,\eta},\mu_{u})\to 0$ when $\eta\to0$ and similarly for all $n\in\bbN$, ${\bf d}(\mu^n_{u,\eta},\mu^n_{u})\to 0$ when $\eta\to0$. 
 Fix $\varepsilon>0$. For $n\in\bbN$, we set 
$$A_n=\{u\in U\left| \exists\eta_0>0\ \forall k\geq n\ \forall \eta\in (0,\eta_0); {\bf d}(\mu^k_{u,\eta},\mu_u)\leq\varepsilon\right.\}.$$
 Then  $\nu(\cup A_n)=1$ and $A_{n}\subseteq A_{n+1}$. Thus there is $n$ such that $\nu(A_n)\geq 1-\varepsilon$ and for $u\in A_n$, ${\bf d}(\mu_u,\mu_u^k)\leq \varepsilon$ for all $k\geq n$. 
This shows that the image sequence of $(\mu^n)$ in $\LL^0_{p_\bullet}(U,\Prob_U(V))$ indeed converges to the image of $\mu$.
\end{proof}

We note that if $G$ acts on the fibration $V\to U$
(that is, $G$ acts on $U$ and $V$ and $p$ is equivariant)
then the disintegration homeomorphism is also equivariant
with respect to the natural action of $G$ on $\LL^0_p(U,V)$
given by $(g f)(u)=g(f(g^{-1}u))$.

\begin{lemma}\label{stab}
Let $p\colon V\to U$ be a Polish fibration  with an action of $G$ such that the $G$-action on $U$ is trivial.  Let $\nu$ be a  probability measure on $U$, and let $f\in\LL^0_p(U,V)$. Then there exists $U_1\subset U$ of full measure such that $$\Stab(f)=\bigcap_{u\in U_1} \Stab(f(u)).$$
\end{lemma}

\begin{proof}

Let $L$ be the stabilizer of $f$ in $G$. If $L'$ is a countable dense subgroup of $L$, then there is a full measure subset $U_1\subset U$ such that $L'\subset \bigcap_{u\in U_1}\Stab(f(u))$ (for any $g\in L'$, there is such a subspace $U_g$. Choose $U_1$ to be the intersection over $L'$). Since all these stabilizers are closed, and $L'$ is dense in $L$, we actually have $L\subset \bigcap_{u\in U_1}\Stab(f(u))$. Since the reverse inclusion is clear, we conclude that 
$$L= \bigcap_{u\in U_1}\Stab(f(u)).$$
\end{proof}

\begin{cor} \label{ec-replacement}
Assume $G$ acts continuously on the Polish space $V$ and the quotient topology on $G\backslash V$ is $T_0$.
Let $L<G$ be a closed subgroup and $\mu$ be an $L$-invariant probability measure on $V$.
Then there exist a point $v\in V$ and an $L$-invariant probability measure on $G\cdot v\simeq G/\Stab(v)$.
\end{cor}

\begin{proof} 
Let $\nu$ be the pushforward measure of $\mu$ on $U$. 
By Theorem~\ref{disintegration}, we may consider the disintegration of $\mu$ as an element $(\mu_u) \in \LL^0_{p_\bullet}(U,\Prob_U(V))$ and this element is clearly $L$-invariant. 
By Lemma \ref{stab}, the stabilizer of $(\mu_u)$ is an intersection of stabilizers of the measures $\mu_u$, for $u$ in a subset of $U$. In particular $L$ stabilizes some $\mu_u$, which is a measure supported on an orbit $G\cdot v$.
The latter is equivariantly homeomorphic to $G/\Stab(v)$ thanks to Effros theorem (Theorem~\ref{effros}).
\end{proof}

\section{Almost algebraic groups and actions}

The goal of this section is the proof of Theorem \ref{mainthm}. Starting with an almost algebraic action of $G$ on a Polish  $V$, we aim to prove that the action $G\action\Prob(V)$ is algebraic as well. So we have to prove that stabilizers of probability measures on $V$ are almost algebraic and the quotient $G\backslash\Prob(V)$ is $T_0$.
Going toward wider and wider generality, we prove the first point in \S\ref{sec:almalg} and the second one in \S \ref{sec:T0}
\subsection{Almost algebraic groups}

Recall that by our setup~\ref{setup}, $(k,|\cdot|)$ is a fixed complete and separable valued field and ${\bf G}$ is a fixed $k$-algebraic group.
By Proposition~\ref{polishing}, $G={\bf G}(k)$ has the structure of a Polish group.
Recall that a closed subgroup $L<G$ is called \emph{almost algebraic} if there exists a $k$-algebraic subgroup ${\bf H}<{\bf G}$
such that $L$ contains $H={\bf H}(k)$ as a normal cocompact subgroup (Definition~\ref{aag-aaa}).

\begin{lemma}\label{intalm}
An arbitrary intersection of almost algebraic subgroups is again almost algebraic. 

More precisely, let $(L_i)_{i\in I}$ be a collection of almost algebraic subgroups and ${\bf H}_i$ algebraic subgroups such that $H_i={\bf H}_i(k)$ is normal and cocompact in $L_i$. 

Then one can find a  finite subset $I_0$ such that, defining ${\bf H}=\cap_{i\in I_0}{\bf H}_i$, we have that  ${\bf H}=\cap _{i\in I}{\bf H}_i$ and ${\bf H}(k)$ is normal and cocompact in $\cap_{i\in I} L_i$.
\end{lemma}

\begin{proof}
Let $L=\cap L_i$ and $H=\cap H_i$ which coincides with $(\cap_{i\in I} {\bf H}_i)(k)$. Then it is straighforward to check that $H\lhd L$. Thanks to the Noetherian property of {\bf G}, there exists a finite subset $I_0\subset I$ such that $\cap_i\mathbf{H}_i$ coincides with $\cap_{i\in I_0} \mathbf{H}_i$. 

Let $\bf L$ be the Zariski closure of $L$ and ${\bf L}_i$ the one of  $L_i$. The diagonal image of ${\bf L}(k)$ in $\prod_{i\in I_0} {\bf L}_i(k)/H_i$ is locally closed by Proposition \ref{polishing} and it is a group. Thus it is actually closed. Moreover it is homeomorphic to ${\bf L}(k)/H$.  To conclude, it suffices to observe that $L/H$ is closed in ${\bf L}(k)/H$ and lies in $\left({\bf L}(k)/H\right)\bigcap\left( \prod_{i\in I_0} L_i/H_i\right)$ which is compact.
\end{proof}

\begin{remark}\label{canonic}
Actually the proof of this lemma shows that any almost algebraic subgroup $L$ has a minimal subgroup among all cocompact normal subgroups $N$ which can be written $N={\bf N}(k)$ for some algebraic subgroup  $\bf N\leq G$. This group is actually the intersection of all such subgroups and it is invariant under the normalizer $N_G(L)$ of $L$ in $G$.
\end{remark}

\begin{lemma}\label{extalm}Let $H,L$ be closed subgroups of $G$ such that $H$ is almost algebraic, $H\lhd L$ and $L/H$ is compact. Then $L$ is almost algebraic.
\end{lemma}

\begin{proof} There is a algebraic subgroup $\bf N$ of $\bf G$ such that $N={\bf N}(k)$ is normal and cocompact in $H$. Moreover thanks to Remark \ref{canonic}, $N$ may be choosen to be invariant under $N_G(H)$ and thus $N$ is cocompact and normal in $L$. 
\end{proof}

\subsection{Almost algebraicity of stabilizers of probability measures}\label{sec:almalg}

Let $V$ be a Polish space endowed with a continuous $G$-action. 
Recall that the action $G\action V$ is called \emph{almost algebraic} if the stabilizers are almost algebraic subgroups of $G$ and the quotient topology on $G\backslash V$ is $T_0$  (Definition~\ref{aag-aaa}).

\begin{remark} \label{remarkeffros}
For a continuous action of $G$ on a Polish space $V$,
the action is almost algebraic if and only if the stabilizers are almost algebraic and for every $v\in V$
and any sequence $g_n\in G$, $g_nv\to v$ implies $g_n \to e$ in $G/\Stab_G(v)$.
This equivalent definition is much easier to check, and we will allow ourselves to use it freely in the sequel.
The two definitions are indeed equivalent by Effros' Theorem~\ref{effros}.
\end{remark}

\begin{example} \label{aaaex}
Let ${\bf I}$ be a $k$-algebraic group and $\phi:{\bf G} \to {\bf I}$ a $k$-morphism.
Let $L$ be an almost algebraic group in $I={\bf I}(k)$.
Then the action of $G$ on $I/L$ is almost algebraic. This fact is proved after Lemma \ref{cnsa}.
\end{example}

\begin{lemma}\label{tameproduct}
Let $K$ be a compact group acting continuously on a $T_0$-space $X$. Then the orbit space $K\backslash X$ is $T_0$ as well.
\end{lemma}

\begin{proof}Continuity of the action means that the action map $K\times X\to K\times X$ which associates $(k,kx)$ to $(k,x)$ is a homeomorphism. Compactness of $K$ implies that the projection $(k,x)\mapsto x$ from $K\times X$ to $X$ is closed. Composing the two yields closedness of the map $(k,x)\mapsto kx$. This implies that if $F\subset X$ is closed, then $K F$ is again closed.

Let $x,y\in X$  in different $K$-orbits. Let us consider $Y=Kx\cup Ky$ with the induced topology. This is a compact $T_0$-space. Now, consider the set of closed non-empty subspaces of $Y$ with the order given by inclusion. By compactness any decreasing chain has a non-empty intersection and thus Zorn's Lemma implies there are minimal elements, that are points since $Y$ is $T_0$. Thus $Y$ has at least a closed point.

Without loss of generality we may and shall assume that $\{x\}$ is closed in $Y$. This means that there exists a closed subset $F$ of $X$ such that $F\cap Y=\{x\}$. In particular $F\cap Ky=\emptyset$, and therefore $Ky\cap KF=\emptyset$. Finally, $KF$ is a closed $K$-invariant set separating $Kx$ from $Ky$. 
\end{proof}
\begin{lemma}\label{cnsa}
Let $J$ be a topological group acting continuously on a topological space $X$. If $N$ is a closed normal subgroup of $J$, the induced action of $J/N$ on $N\backslash X$ is continuous and the orbits spaces $J\backslash X$ and $(J/N)\backslash (N\backslash X)$ are homeomorphic. 
\end{lemma}

\begin{proof}The map $(g,x)\mapsto Ngx$ from $J\times X$ to $N\backslash X$ is continuous and goes through the quotient space $J/N\times N\backslash X$ which is the orbit space of $N\times N$ acting diagonally on $J\times X$. Thus, $(gN,Nx)\mapsto Ngx$ is continuous, that is the action of $J/N$ on $N\backslash X$ is continuous. 

By the universal property of the topological quotient, the continuous map $x\mapsto(J/N)Nx$ from $X$ to $(J/N)\backslash (N\backslash X)$ induces a continuous map $J\backslash X \to (J/N)\backslash (N\backslash X)$. Conversely, the continuous map $N\backslash X\to J\backslash X$ induces also a continuous map $(J/N)\backslash (N\backslash X)\to J\backslash X$ which is the inverse of the previous one.
\end{proof}

\begin{proof}[Proof of Example \ref{aaaex}]
Since $\phi^{-1}(L)$ and its conjugates are almost algebraic in $G$, it is clear that the stabilizers are almost algebraic.
So we are left to prove that the topology on $G\backslash I/L$ is $T_0$. Let $H$ be a cocompact normal subgroup in $L$ with $H={\bf H}(k)$ for some $k$-algebraic subgroup ${\bf H}$ of ${\bf I}$.
By Lemma \ref{cnsa} the orbit space $G\backslash I/L$ is homeomorphic to the space of orbits of the action of $G\times (L/H)$ on $I/H$. 
Note that the action of $G$ on $I/H\subset {\bf I}/{\bf H}(k)$ has locally closed orbits (and therefore $G\backslash I/H$ is $T_0$) by Proposition~\ref{polishing}, as the action of ${\bf G}$ on ${\bf I}/{\bf H}$ is $k$-algebraic. Now the $T_0$ property of $G\backslash I/L$ follows from Lemma~\ref{tameproduct} for the compact group $L/H$ acting continuously on the $T_0$-space $G\backslash I/H$.
\end{proof}

\begin{lemma}\label{product}
Let $J$ be a countable set, $(L_i)_{i\in J}$ a family of almost algebraic subgroups of $G$. Then the diagonal action of $G$ on $\prod_{i\in J}G/L_i$ is almost algebraic. 
\end{lemma}

\begin{proof}
Stabilizers of points in $\prod_{i\in J}G/L_i$ are intersections of almost algebraic subgroups of $G$. Hence by Lemma \ref{intalm} they are almost algebraic. So we just have to prove that $G\backslash \left(\prod_{i\in J}G/L_i\right)$ is $T_0$.

 For $i\in J$, let $\bfH_i$ be an algebraic subgroup of $\bfG$ such that $H_i=\bfH_i(k)$ is a cocompact normal subgroup of $L_i$. Consider $V=\prod_{i\in J} G/H_i$. We first prove that the topology on $G\backslash V$ is $T_0$, by proving that orbit maps are homeomorphisms (Theorem~\ref{effros}).  Let $(h_iH_i)_{i\in J}$ be an element of $V$ and $(g_n)$ be a sequence of elements of $G$ such that $g_n \cdot (h_i H_i)$ converges to $(h_iH_i)$ in $V$. 
 
 Let $H=\bigcap_{i\in J} h_iH_ih_i^{-1}=\Stab((h_iH_i)_{i\in J})$. We have to prove that $g_n$ converges to $e$ in $G/H$ (see Remark \ref{remarkeffros}). By Noetherianity, there exists a finite $J_0\subset J$ such that $H=\bigcap_{i\in J_0}h_i H_ih_i^{-1}$. Set $V_0=\prod_{i\in J_0} G/H_i$. We see that, in $V_0$, we have that $g_n. (h_iH_i)_{i\in J_0}$ converges to $(h_iH_i)_{i\in J_0}$. By Proposition \ref{polishing}, it follows that $g_n$ converges to the identity in $G/H$.
 
 Now let $K$ be the compact group $\prod_{i\in J}L_i/H_i$. The group $K$ acts also continuously on $V$ via the formula $(l_iH_i)\cdot(g_iH_i)=(g_il^{-1}_iH_i)$ and this action commutes with the action of $G$.
Thus we can apply Lemma \ref{tameproduct} to $K$ acting on $G\backslash V$ and get that the space of orbits for the $G$-action on $V/K\simeq \prod_{i\in J}G_i/L_i$ is $T_0$, as desired.
 \end{proof}

Our main goal in this subsection is proving the following theorem, which is an essential part of our main theorem, Theorem~\ref{mainthm}.

\begin{theorem}\label{almost-alg}

Let $V$ be a Polish space with an almost algebraic action of $G$. Then stabilizers of probability measures on $V$ are almost algebraic subgroups of $G$.
\end{theorem}

We first restate and prove Proposition~\ref{prop:stab}, discussed in the introduction.

\begin{prop} 
Fix a closed subgroup $L<G$.
Then there exists a $k$-subgroup ${\bf H}_0<{\bf G}$ which is normalized by $L$
such that $L$ has a precompact image in the Polish group $(N_{\bf G}({\bf H}_0)/{\bf H}_0)(k)$
and such that for every $k$-${\bf G}$-variety ${\bf V}$, any $L$-invariant finite measure on ${\bf V}(k)$
is supported on the subvariety of ${\bf H}_0$-fixed points.
\end{prop}

\begin{proof}
Replacing $\bfG$ by the Zariski closure of $L$, we assume that $L$ is Zariski-dense in ${\bf G}$
and consider the collection 
\[ \{{\bf H}<{\bf G}~|~{\bf H} \mbox{ is a $k$-algebraic subgroup},~\Prob({\bf G}/{\bf H}(k))^L\neq \emptyset\}. \]
By the Noetherian property of ${\bf G}$ there exists a minimal element ${\bf H}_0$ in this collection.
We let $\mu_0$ be a corresponding $L$-invariant measure on ${\bf G}/{\bf H}_0(k)$.

We first claim that ${\bf H}_0$ is normal in ${\bf G}$.
Assuming not, we let ${\bf N} \lneq {\bf G}$ be the normalizer of ${\bf H}_0$ and consider the set 
\[{\bf U} =\{(x{\bf H}_0,y{\bf H}_0)~|~y^{-1}x\notin {\bf N}\} \subset {\bf G}/{\bf H}_0\times {\bf G}/{\bf H}_0. \]
This set is a non-empty Zariski-open set which is invariant under the diagonal ${\bf G}$-action,
as its complement is the preimage of the diagonal under the natural map ${\bf G}/{\bf H}_0\times {\bf G}/{\bf H}_0 \to {\bf G}/{\bf N}\times {\bf G}/{\bf N}$.
Since the support of $\mu_0\times\mu_0$ in ${\bf G}/{\bf H}_0\times {\bf G}/{\bf H}_0$ is invariant under $L\times L$ which is Zariski-dense in ${\bf G}\times {\bf G}$ we get that $(\mu_0\times\mu_0)({\bf U}(k))\neq 0$.
It follows from Corollary~\ref{ec-replacement} that there exist $u\in {\bf U}(k)$ and an $L$-invariant finite measure on $G/\Stab_G(u)\subset ({\bf G}/\Stab_{\bf G}(u))(k)$.
By the definition of ${\bf U}$ we get a contradiction to the minimality of ${\bf H}_0$,
as point stabilizers in ${\bf U}$ are properly contained in conjugates of ${\bf H}_0$.
This proves that ${\bf H}_0$ is normal in ${\bf G}$.

Next we let ${\bf V}$ be a $k$-${\bf G}$-variety and $\mu$ be an $L$-invariant measure on ${\bf V}(k)$.
We argue to show that $\mu$ is supported on ${\bf V}^{{\bf H}_0}\cap {\bf V}(k)$.
Indeed, assume not. Let $\bf{V'}$ be the Zariski-closure of ${\bf V}(k)\cap {\bf V}^{\bf H_0}$, and ${\bf V}''={\bf V}-{\bf V'}$. Then we see that ${\bf V'}$ is defined over $k$ \cite[AG, 14.4]{borel}. Furthermore, $\bf H_0$ acts on $\bf V'$ trivially, so that we have  ${\bf V'}(k)={\bf V}(k)\cap {\bf V}^{\bf H_0}$. 
Hence by assumption we get that $\mu({\bf V}''(k))>0$. Replacing ${\bf V}$  by ${\bf V}''$ and restricting and normalizing the measure,
 we may and shall assume that ${\bf V}^{{\bf H}_0}\cap {\bf V}(k)=\emptyset$.
 
We consider the variety ${\bf G}/{\bf H}_0\times{\bf V}$ as a $k$-${\bf G}$-variety.
The measure $\mu_0\times \mu$ is an $L$-invariant measure on $({\bf G}/{\bf H}_0\times{\bf V})(k)$.
It follows from Corollary~\ref{ec-replacement} that there exists $u\in ({\bf G}/{\bf H}_0\times{\bf V})(k)$ and an $L$-invariant measure on 
$G/\Stab_G(u)$.
By Proposition~\ref{polishing} there exist a $k$-algebraic subgroup ${\bf H}<{\bf G}$ with $H={\bf H}(k)=\Stab_G(u)$ and an orbit map 
${\bf G}/{\bf H}\to {\bf G}u$ inducing a homeormorphism $G/H\to G/\Stab_G(u)$.
Thus we obtain an $L$-invariant probability measure on ${\bf G}/{\bf H}(k)$.
Now, $\bf H$ is contained in some conjugate $g{\bf H}_0g^{-1}$, for some $g\in {\bf G}$.  Hence we get that $g^{-1}{\bf H} g< {\bf H}_0$ is such that ${\bf G}/ g^{-1}{\bf H} g$ has an $L$-invariant probability measure. By minimality, this implies that $g^{-1}{\bf H}g={\bf H_0}$, hence by normality of $\bf H_0$, ${\bf H}={\bf H_0}$. Therefore $u$ belongs to ${\bf V}(k)\cap {\bf V}^{\bf H_0}$, which was assumed to be empty. Hence we get a contradiction. 
This proves that $\mu$ is supported on ${\bf V}^{{\bf H}_0}$.

We set $S=({\bf G}/{\bf H}_0)(k)$
and let $T$ be the closure of the image of $L$ in $S$.
We are left to show that $T$ is compact.
$S$ is a Polish group and $T$ is a closed subgroup.
The quotient topology on $T\backslash S$ is Hausdorff, and in particular $T_0$.
The measure $\mu_0$ is an $L$-invariant finite measure on $S$, hence it is also $T$-invariant.
Substituting $S=V$ and $T=G=L$ in Corollary~\ref{ec-replacement} we find a finite measure $\mu_1$ on $S$ which is supported on a unique $T$-coset, $Ts$. 
The measure $(R_s)_*\mu_1$, given by pushing $\mu_1$ by the right translation by $s^{-1}$ is then a $T$-invariant probability measure on $T$.
It is well-known result due to A. Weil (see \cite{Oxtoby} where the result is attributed to Ulam) that a Polish group that admits an invariant measure class is locally compact, and a locally compact group that admits an invariant probability measure is compact. Thus $T$ is indeed compact.
\end{proof}

\begin{cor}\label{almProbHom}
Fix a $k$-${\bf G}$-algebraic variety ${\bf V}$, and set $V={\bf V}(k)$. Let $\mu\in\Prob(V)$. Then $\Stab(\mu)$ is almost algebraic.
\end{cor}

\begin{proof}
Let $L=\Stab(\mu)$. We may and shall assume $L$ to be Zariski-dense in $G$, and we can find $\bfH_0$ as in Proposition \ref{prop:stab}.
We know that $\mu$ is supported on the set of ${\bf V}^{{\bf H}_0}$ thus $H_0={\bf H}_0(k)<L$.
Since $G/H_0$ is acting on ${\bf V}^{{\bf H}_0}\cap {\bf V}(k)$
and the stabilizer of $\mu$ is closed in $G/H_0$, we conclude that $L$ has a closed image.
We know that the image of $L$ is precompact, thus it is actually compact, and we conclude that $L$ is almost algebraic.
\end{proof}

\begin{lemma}\label{almHaar}
Let $L<G$ be an almost algebraic group, with $H=\bfH(k)$ a normal cocompact algebraic subgroup of $L$.
Then there is a $G$-equivariant continuous map $\phi:\Prob(G/L)\to \Prob(G/H)$. Furthermore, we have, for every $\mu\in \Prob(G/L)$, $\Stab(\mu)=\Stab(\phi(\mu))$.
\end{lemma}

\begin{proof}
Let $\lambda$ be a Haar probability measure on $L/H$. For a continuous bounded function $f$ on $G/H$ let $\overline f$ be the continuous bounded function on $G/L$ defined by $\overline f(gL)=\int_{L/H}f(gh)\, \dd\lambda(h)$ and finally $\phi(\mu)(f)=\mu({\overline f})$.

Then it is clear that $\phi$ is equivariant, and we deduce that $\Stab(\mu)\subset \Stab(\phi(\mu))$. In the other direction, we note that if $\pi:G/H\to G/L$ is the projection, we have $\pi_*(\phi(\mu))=\mu$. Hence the other inclusion is also clear.

To check the continuity, let $\mu_n\to \mu\in \Prob(G/L)$, and take $f$ a continuous bounded function on $G/H$. Then $\phi(\mu_n)(f)=\mu_n(\overline f) \to \mu(\overline f)=\phi(\mu)(f)$. Hence $\phi(\mu_n)$ converges to $\phi(\mu)$.
\end{proof}

\begin{proof}[Proof of Theorem~\ref{almost-alg}] 
Choose $\mu\in\Prob(V)$ and denote $L=\Stab_G(\mu)$, $H=\Fix_G(\supp(\mu))$. 
Set $U=G\backslash V$, and let $\nu=p_\bullet\mu$, where $p:V\to U$ is the projection.
Note that $p$ is a Polish fibration.
By Theorem~\ref{disintegration}, $L$ is equal to the stabilizer of an element $f\in\LL^0_{p_\bullet}(U,\Prob_U(V))$.
By Lemma \ref{stab} there exists a $\nu$-full measure set $U_1\subset U$ such that $L=\bigcap_{u\in U_1}\Stab(f(u))$. 
For a fixed $u\in U_1$, $f(u)$ is a measure on a $G$-orbit in $V$ which we identify with $G/L'$ for some almost algebraic subgroup $L'<G$. 
Let ${\bf H'}<{\bf G}$ be a $k$-algebraic subgroup such that $H'={\bf H'}(k)$ is a cocompact normal subgroup of $L'$. 
By Lemma \ref{almHaar}, $\Stab(f(u))$ is also the stabilizer of a probability measure on $G/H' \subset {\bf G}/{\bf H'}(k)$. 
By Corollary \ref{almProbHom}, it follows that $\Stab(f(u))$ is almost algebraic.
We conclude that $L$ is almost algebraic by Lemma~\ref{intalm}.
\end{proof}

\subsection{Separating orbits in the space of probability measures}\label{sec:T0}

In this subsection, we prove the following theorem.

\begin{theorem}\label{tameomog}
Let $L<G$ be an almost algebraic subgroup. Then the action of $G$ on $\Prob(G/L)$ is almost algebraic.
\end{theorem}

The proof of Theorem~\ref{tameomog} consists in several steps, proving particular cases of the theorem, each of them using the previous one. First we start with the case when $L=G$ (Lemma \ref{prob(G)}). Then we treat the case when $L$ is a normal algebraic subgroup of $G$ (Lemma \ref{prob(G/N)}). The main step is then to deduce the theorem when $L$ is any algebraic subgroup of $G$ (Proposition \ref{aaaa}), before concluding with the general case.

\begin{lemma} \label{prob(G)}
The $G$-action on $\Prob(G)$ is almost algebraic.
\end{lemma}

\begin{proof}
The regular action of $G$ on itself is proper, so by Lemma~\ref{Prokhorov} it follows that the action of $G$ on $\Prob(G)$ is proper. Any proper action is almost algebraic.
\end{proof}

\begin{lemma} \label{prob(G/N)}
Let ${\bf H}<{\bf G}$ be a normal $k$-algebraic subgroup.
Then the $G$-action on $\Prob(({\bf G}/{\bf H})(k))$ is almost algebraic.
\end{lemma}

\begin{proof}
Denoting ${\bf I}={\bf G}/{\bf H}$ and $I={\bf I}(k)$,
we know that the $I$-action on $\Prob(I)$ is almost algebraic (Lemma \ref{prob(G)}). Since $G/H$ is a subgroup of $I$, $G$ stabilizes each $I$-orbit. It is thus enough to show that $G$ acts almost algebraically on each $I$-orbit.
We know that such an orbit is of the form $I/L$ where $L$ is almost algebraic (Theorem~\ref{almost-alg}),
so this follows from Example~\ref{aaaex}.
\end{proof}

An essential technical tool for proving Theorem~\ref{tameomog} and Theorem~\ref{mainthm} is given by the following proposition.

\begin{proposition}\label{tame-orbits}
Let  $V$ be a Polish space, with a continuous action of $G$. Assume that 
\begin{itemize}
\item The quotient topology on $G\backslash V$ is $T_0$, and
\item For any $v\in V$, the action of $G$ on $\Prob(G.v)$ is almost algebraic.
\end{itemize}
Then the quotient topology on $G\backslash \Prob(V)$ is $T_0$.
\end{proposition}

The proposition will directly follow from the following lemma.

\begin{lemma}\label{tame-sections}
Let $p:V\to U$ be a Polish fibration with an action of $G$, and let $\nu$ be a  probability measure on $U$. Assume that the action of $G$ on $U$ is trivial and that the action of $G$ on $\Prob(p^{-1}(u))$ is almost algebraic for almost every $u\in U$. Let $P=\{\mu\in \Prob(V)\mid p_*\mu=\nu\}$. Then the topology on $G\backslash P$ is $T_0$.
\end{lemma}

This proof is similar to the proof presented in \cite[Proof of Proposition 3.3.1]{zimmer-book}; see also \cite[Lemma 6.7]{Burger-Acampo}.

\begin{proof}
The set $P$ is Polish, as a closed subset of $\Prob(V)$.
By Theorem~\ref{effros} we need to show that the orbit maps are homeomorphisms. By Theorem~\ref{disintegration}, $P$ is equivariantly homeomorphic to $\LL^0_{p_\bullet}(U,\Prob_U(V))$.

Fixing $f\in \LL^0_{p_\bullet}(U,\Prob_U(V))$ and letting $g_n\in G$ be such that $g_nf \to f$, we will show that $g_n$ converges to the identity in $G/\Stab(f)$ by proving that every subsequence of $(g_n)$ has a sub-subsequence which converges to the
identity in $G/\Stab(f)$. 
Doing so, we are free to replace $(g_n)$ by any subsequence.
Relying on Lemma~\ref{convsubseq}, we replace $(g_n)$ by a subsequence such that $g_n f(u)\to f(u)$ for every $u$ in some $\nu$-full subset $U_0\subset U$. Let $U_1\subset U_0$ be a full measure subset such that the action of $G$ on $p^{-1}(u)$ is almost algebraic for every $u\in U_1$.

Let $u\in U_1$. By definition, we know that $f(u) \in \Prob(p^{-1}(u))$ and that the action of $G$ on $\Prob(p^{-1}(u))$ is almost algebraic.
By Proposition~\ref{polishing}, the orbit map $G/\Stab(f(u)) \to Gf(u)$ is a homeomorphism thus $g_nf(u)\to f(u)$ implies that $g_n$ converges to the identity in $G/ \Stab(f(u))$. 
By Lemma \ref{stab}, there is also a full measure subset $U_2$, that we may and do assume to be contained in $U_1$, such that 
$$\Stab(f)=\bigcap_{u\in U_2} \Stab(f(u))$$
and since $G$ is second countable, one can find $U_3$ countable in $U_2$ such that 
$$\Stab(f)=\bigcap_{u\in U_3} \Stab(f(u)).$$
By assumption, for every $u\in U_3$, the group $\Stab(f(u))$ is almost algebraic. Hence by Lemma~\ref{product}, the action of $G$ on $\prod_{u\in U_3} G/\Stab(f(u))$ is almost algebraic. In particular, we see that  $g_n$ converges to $e$ in $G/\Stab(f)$. 
\end{proof}

\begin{proof}[Proof of Proposition \ref{tame-orbits}]
Let $U=G\backslash V$ and $p\colon V\to U$ be the projection. 
Consider the $G$-invariant continuous map $p_*\colon\Prob(V)\to \Prob(U)$.
Clearly the fibers of $p_*$ are closed and $G$-invariant, so it is enough to prove that for a given $\nu \in \Prob(U)$, the quotient space  $G\backslash p_*^{-1}(\{\nu\})$ has a $T_0$ -topology. This is precisely  Lemma \ref{tame-sections}.
\end{proof}

Let $\pi:V\to V'$ be a continuous $G$-map between Polish spaces, $\mu\in\Prob(V)$ and $\nu=\pi_*\mu$. 
Then $\nu$ has a unique decomposition $\nu=\nu_c+\nu_d$ where $\nu_c$ and $\nu_d$ are the continuous and discrete parts of $\nu$. 
Moreover $\nu_d$ can be written $\sum_{\lambda\in\Lambda}\lambda\sum_{f\in F_\lambda} \delta_f$,
where $\Lambda=\{\lambda\in\bbR_+~|~ \exists u\in V',\ \pi_*\mu(\{u\})=\lambda\}$ and $F_\lambda=\{u\in V'~|~ \nu(\{u\})=\lambda\}$. 
Defining $\mu_\lambda$ to be the restriction of $\mu$ to $\pi^{-1}(F_\lambda)$ and $\mu_c=\mu-\sum_{\lambda\in\Lambda}\mu_\lambda$, 
we have a unique decomposition $\mu=\mu_c+\sum_{\lambda\in \Lambda} \mu_\lambda$, 
where $\pi_*(\mu_c)$ is non-atomic and each $\pi_*(\mu_\lambda)$ is a finitely supported, uniform measure of the form $\lambda\sum_{f\in F_\lambda} \delta_f$.

\begin{lemma} \label{decom}
Let $\pi:V\to V'$ be a continuous $G$-map between Polish spaces and $\mu\in\Prob(V)$. Using the above decomposition, we have
$\Stab(\mu)=\Stab(\mu_c)\cap\left(\bigcap_\lambda \Stab(\mu_\lambda)\right)$.
If $g_n\mu\to \mu$ then $g_n\mu_c\to \mu_c$ and for each $\lambda\in\Lambda$, $g_n\mu_\lambda\to \mu_\lambda$.
\end{lemma}

\begin{proof} The statement about $\Stab(\mu)$ is straightforward from the uniqueness of the decomposition of $\mu$. Let $(g_n)$ be a sequence such that $g_n\mu\to \mu$. Once again, we use a sub-subsequence argument: we prove that any subsequence of $(g_n)$ contains a sub-subsequence such that $g_n\mu_\lambda\to \mu_\lambda$ for every $\lambda$. Hence we start by replacing $(g_n)$ by an arbitrary subsequence.

Observe that $g_n\mu\to \mu$ implies $g_n\nu\to\nu$ because $\pi_*\colon\Prob(V)\to\Prob(V')$ is continuous.  
Let $K'$ be a compact metrizable space in which $V'$ is continuously embedded as a $G_\delta$-subset (see \cite[Theorem 4.14]{kechris}). 
Then $\Prob(V')$ embeds as a $G_\delta$-subset in $\Prob(K')$ as well \cite[Proof of Theorem 17.23]{kechris}. 
We begin with the following observation. 
Assume $\nu_n$ is a sequence of probability measures converging to $\nu\in\Prob(V')$ and $\nu_n$ decomposes as $\delta_{u_n}+\nu'_n$ with $u_n\in V'$ and $\nu'_n\in\Prob(V')$. 
Up to extraction $u_n$ converges to some $k\in K'$ and thus $\nu(\{k\})>0$ which implies that $k\in V'$. 

Let $\lambda_1$ be the maximum of $\Lambda$. 
The above observation implies that up to extraction we may assume that for any $f\in F_{\lambda_1}$, $g_nf$ converges to some $l(f)\in V'$. 
Since $g_n\nu\to\nu$, we have that $l(f)\in F_{\lambda_1}$ thus $g_n\nu_{\lambda_1}$ converges to $\nu_{\lambda_1}$, where $\nu_{\lambda_1}=\pi_*(\mu_{\lambda_1})$. An induction on $\Lambda$ (countable and well ordered with the reverse order of $\bbR$) shows that (after extraction) $g_n\nu_\lambda\to\nu_\lambda$ for any $\lambda\in\Lambda$. 

Once again, we embed $V$ in some compact metrizable space $K$. Fix $\lambda\in\Lambda$ and let $\mu'$ be an adherent point of $(g_n\mu_\lambda)$ in $\Prob(K)$. 
As $\pi$ is $G$-equivariant, we have that $\pi_*g_n\mu_\lambda=g_n\pi_*\mu_\lambda=g_n\nu_\lambda$ which converges to $\nu_\lambda$. Hence $\pi_*\mu'=\nu_\lambda$. Furthermore, we also see that $\mu'$ is supported on $\pi^{-1}(F_\lambda)$, hence $\mu'\in \Prob(V)$.

The same argument proves that $\mu-\mu'$, which is an adherent point of 
 $g_n(\mu-\mu_\lambda)$,  is supported on $V\setminus\pi^{-1}(F_\lambda)$.
 
 As $\mu$ can be written uniquely as a sum of a measure supported on $\pi^{-1}(F_\lambda)$ and a measure supported on $V\setminus\pi^{-1}(F_\lambda)$, we see,
  writing $\mu=(\mu-\mu')+\mu'=(\mu-\mu_\lambda)+\mu_\lambda$,
 that necessarily $\mu'=\mu_\lambda$. 
This concludes the proof since $\mu_c=\mu-\sum_{\lambda\in\Lambda}\mu_\lambda$.
\end{proof}

\begin{lemma}\label{lem:atom}
Let ${\bf H}<{\bf G}$ be a $k$-algebraic subgroup.
Set ${\bf N}=N_{\bf G}({\bf H})$, $H={\bf H}(k)$ and $N={\bf N}(k)$.
Let ${\bf V}={\bf G}/{\bf H}$, ${\bf V}'={\bf G}/{\bf N}$, $V={\bf V}(k)$ and $V'={\bf V}'(k)$.
Consider the map $\pi:V\to V'$.
Let $F\subset V'$ be a finite set, $\nu=1/|F|\sum_{f\in F} \delta_f$ and $\mu\in \Prob(V)$ be a measure with $\pi_*\mu=\nu$.
Let $(g_i)$ be a sequence with $g_i\mu\to\mu$. Then $g_i\to e\in G/\Stab(\mu)$.
\end{lemma}

\begin{proof}
Denote $m=|F|$. 
We know that $({\bf V'})^m/\Sym(m)$ is an algebraic variety,
hence by Proposition~\ref{polishing}, every $G$-orbit in $(V')^m/\Sym(m)$ is locally closed.
It follows in particular that $g_i\to e$ in $G/\Stab(F)$.

 Again, it is enough to show that every subsequence of $(g_i)$ contains a subsequence which tends to $e$ modulo $\Stab(\mu)$. We start by extracting an arbitrary subsequence of $(g_i)$.

 Let us number $f_1,f_2,\dots,f_m$ the elements of $F$ and denote $F'=(f_1,\dots, f_m)\in (V')^m$. 
 Since $g_i$ converges to $e$ in $G/\Stab(F)$, it follows that, passing to a subsequence, there exists $\sigma\in\Sym(m)$ such that $g_iF'$ tends to $\sigma(F')=(f_{\sigma(1)},\dots, f_{\sigma(m)})$.
 This means $\overline{GF'}\supset G\sigma(F')$ 
 and thus $\overline{GF'}\supset \overline{G\sigma(F')}\supset\dots\supset\overline{G\sigma^n(F')
 }=\overline{GF'}$ for some $n\in\bbN$.
  In particular $\overline{GF'}=\overline{G\sigma(F')}$ and since orbits are locally closed we have that $GF'=G\sigma(F')$.
  
This shows that there exists $g\in \Stab(F)$ such that $gF'=\sigma(F')$. 
Hence we have $g_i F'\to gF'$, and by almost algebraicity of the action on $(V')^m$ it follows that $g_i$ tends to $g$ modulo $\Stab(F')=\bigcap_{f\in F}\Stab(f)$.

Let us fix some notations. For $f\in F$ we denote by $\mu_f$ the restriction of $\mu$ to ${\pi}^{-1}(\{f\})$ and fix $\overline f\in \bf G$ such that $\overline f\mathbf{N}=f$ and denote by $\mathbf{H}_f\leq \bf G$ the conjugate of $\bf H$ by $\overline f$. Observe that $\overline f \mathbf{N}\overline{f}^{-1}=\Stab_{\mathbf{G}}(f)$,  $\mathbf{H}_f\lhd \Stab_{\mathbf{G}}(f)$ and $\boldsymbol{\pi}^{-1}(\{f\})\simeq \Stab_{\mathbf{G}}(f)/\mathbf{H}_f$ where $\boldsymbol{\pi}\colon\bf G/H\to G/N$ is the projection and $\Stab_\mathbf{G}(f)$ is the stabilizer of $f$ under the action of $\bf G$ on $\bf G/N$. We also denote $\mu_f'=g^{-1}\mu_{\sigma(f)}$ and $g'_i=g^{-1}g_i$.
Since $g'_i\to e\in G/\bigcap_{f\in F}\Stab(f)$ there exist $n_i\in \bigcap_{f\in F}\Stab(f)$ such that $g'_in_i^{-1}$ converges to $e$ (in $G$). We observe that $n_i\mu_f = n_i (g'_i)^{-1}g'_i \mu_f$. 
As $g'_i\mu_f$ tends to $\mu'_f$ and $n_i (g'_i)^{-1}$ tends to $e$, we have that $n_i\mu_f$ converges to $\mu'_f$. 

Those measures are supported on $\pi^{-1}(\{f\})\simeq\left(\Stab_{\mathbf{G}}(f)/\mathbf{H}_f\right)(k)$.
By Lemma \ref{prob(G/N)}, $\Stab(f)$ acts almost algebraically on $\Prob\left(\left(\Stab_{\mathbf{G}}(f)/\mathbf{H}_f\right)(k)\right)$. So we have that $n_i$ tends to some $n$ in $\Stab(f)/\Stab(\mu_f)$.

We conclude that $g'_i=g'_in_i^{-1}n_i$ tends to $n$ in $G/\Stab(\mu_f)$. Arguing similarly for every $f$, it follows that $g_i$ tends to $gn$ in $G/\bigcap_{f\in F}\Stab(\mu_f)$. Hence $(g_i)$ converges also in $G/\Stab(\mu)$, since $\bigcap_{f\in F}\Stab(\mu_f)\leq\Stab(\mu)$. Let $h$ be the limit point of $(g_i)$ modulo $\Stab(\mu)$. Then we have that $g_i\mu$ converges to $h\mu$ by continuity of the action. Hence $h\in\Stab(\mu)$, meaning that $h=e$ modulo $\Stab(\mu)$. In other words, $g_i$ converges to $e$ in $G/\Stab(\mu)$.
%
\end{proof}

\begin{prop} \label{aaaa}
Let ${\bf H}<{\bf G}$ be a $k$-algebraic subgroup and set $H={\bf H}(k)$.
Then the action of $G$ on $\Prob(G/H)$ is almost algebraic.
\end{prop}

\begin{proof}
Assume the proposition fails for an algebraic subgroup ${\bf H}$.
We also assume, as we may, that ${\bf H}$ is minimal in the collection of $k$-subgroup of ${\bf G}$ with the property that the $G$-action on $\Prob(G/H)$ is not almost algebraic.
By Theorem~\ref{almost-alg}, $G$ acts on $\Prob(G/H)$ with almost algebraic stabilizers. Hence we have to show that for every measure $\mu\in \Prob(G/H)$ and sequence $g_n$ with $g_n\mu\to\mu$ then $g_n$  tends to $e$ in $G/\Stab(\mu)$ (Remark \ref{remarkeffros}).
We fix such a measure $\mu$ and a sequence $g_n$.
We will achieve a contradiction by showing that $g_n$ does tend to $e$ in $G/\Stab(\mu)$.

We set ${\bf N}=N_{\bf G}({\bf H})$, $N={\bf N}(k)$, ${\bf V}={\bf G}/{\bf H}$, ${\bf V}'={\bf G}/{\bf N}$, $V={\bf V}(k)$ and $V'={\bf V}'(k)$.
We consider the natural inclusion $G/H \subset V$ and view $\mu$ as a measure on $V$.
We consider the projection map $\pi\colon V\to V'$ and set $\nu=\pi_*\mu$.
We use the notation introduced in the discussion before Lemma~\ref{decom}. 
The lemma gives: 
$\Stab(\mu)=\Stab(\mu_c)\cap\left(\bigcap_{\lambda\in\Lambda} \Stab(\mu_\lambda)\right)$ where $\Lambda$ is countable subset of $[0,1]$,
$g_n\mu_c\to \mu_c$ and for each $\lambda\in\Lambda$, $g_n\mu_\lambda\to \mu_\lambda$.
By Lemma~\ref{lem:atom}, for each $\lambda\in\Lambda$,  $g_i\to e\in G/\Stab(\mu_\lambda)$.
Assume given also that $g_n\to e\in G/\Stab(\mu_c)$.
Since by Theorem~\ref{almost-alg} the groups $\Stab(\mu_\lambda)$ and $\Stab(\mu_c)$ are almost algebraic, we will get by Lemma~\ref{product}
that the action of $G$ on $G/\Stab(\mu_c)\times \prod_{\lambda}G/\Stab(\mu_\lambda)$ is almost algebraic. Hence,
\[ g_n\to e \in \left.G\middle/\left(\Stab(\mu_c)\cap\left(\bigcap_\lambda \Stab(\mu_\lambda)\right)\right)\right.=G/\Stab(\mu),\] 
achieving our desired contradiction.
We are thus left to show that indeed $g_n\to e\in G/\Stab(\mu_c)$.

For the rest of the proof we will assume as we may $\mu=\mu_c$, that is $\nu\in \Prob(V')$ is atom-free.
We consider the measure $\mu\times \mu \in \Prob(V\times V)$
and the subset 
\[{\bf U}=\{(x{\bf H},y{\bf H})~|~y^{-1}x\notin {\bf N}\} \subset {\bf G}/{\bf H}\times {\bf G}/{\bf H} ={\bf V}\times {\bf V}\]
defined and discussed in the proof of Proposition~\ref{prop:stab}. We set $U={\bf U}(k)$.
Note that the diagonal in $V'\times V'$ is $\nu\times\nu$-null as $\nu$ is atom-free,
thus $U$ is $\mu\times\mu$-full.
We view as we may $\mu\times \mu$ as a probability measure on $U$.

We now consider the $G$-action on $U$ and claim that the $G$-orbits are locally closed and for every $u\in U$, $G$ acts almost algebraically on $\Prob(Gu)$.
The fact that the $G$-orbits are locally closed follows from Proposition~\ref{polishing}, as ${\bf U}$ is a $k$-subvariety of ${\bf V}$.
Fix now a point $u=(x\bfH,y\bfH) \in U$ for some $x,y\in G$, and consider the $G$-action on $\Prob(Gu)$.
By the definition of ${\bf U}$, ${\bf H}\cap {\bf H}^{y^{-1}x} \lneq {\bf H}$, thus by the minimality of ${\bf H}$ the $G$-action on $\Prob(G/H\cap H^{y^{-1}x})\simeq \Prob(G/H^x\cap H^y)$ is almost algebraic.
Since by Proposition~\ref{polishing} $G/H^x\cap H^y$ is equivariantly homeomorphic to $Gu$ we conclude that indeed, $G$ acts almost algebraically on $\Prob(Gu)$, and the claim is proved.

By Proposition~\ref{tame-orbits}, we conclude that $G$ acts on $\Prob(U)$ almost algebraically.
Hence Effros' Theorem~\ref{effros} implies that $g_n\to e$ in $G/\Stab(\mu\times\mu)$ as $g_n(\mu\times\mu)\to \mu\times \mu$.
Observing that $\Stab(\mu\times\mu)=\Stab(\mu)$, the proof is complete.
\end{proof}

\begin{proof}[Proof of Theorem~\ref{tameomog}]
By Theorem~\ref{almost-alg} we know that the point stabilizers in $\Prob(G/L)$ are almost algebraic.
We are left to show that for every $\mu \in \Prob(G/L)$, for every sequence $g_n \in G$ satisfying $g_n\mu \to \mu$ we have 
 $g_n\to e$ modulo $\Stab(\mu)$
(see Remark~\ref{remarkeffros}).
Fix $\mu \in \Prob(G/L)$ and a sequence $g_n \in G$ satisfying $g_n\mu \to \mu$.
Let ${\bf H}<{\bf G}$ be a $k$-algebraic subgroup with $H=\bfH(k)$ normal and cocompact in $L$,
and recall that by Lemma~\ref{almHaar} we can find 
a $G$-equivariant continous map $\phi:\Prob(G/L)\to \Prob(G/H)$ such that $\Stab(\mu)=\Stab(\phi(\mu))$.
We get that $g_n\phi(\mu)\to \phi(\mu)$.
By Proposition~\ref{aaaa}, the $G$-action on $\Prob(G/H)$ is almost algebraic, thus $g_n\to e$ modulo $\Stab(\phi(\mu))$.
This finishes the proof, as $\Stab(\mu)=\Stab(\phi(\mu))$.
\end{proof}

\subsection{Proof of Theorem~\ref{mainthm}} \label{proofof}

For the convenience of the reader we restate Theorem~\ref{mainthm}.

\begin{theorem} 
If the action of $G$ on $V$ is almost algebraic then the action of $G$ on $\Prob(V)$ is almost algebraic as well.
\end{theorem}

\begin{proof}
By Theorem~\ref{almost-alg}, we know that the $G$-stabilizers in $\Prob(V)$ are almost algebraic.
We need to show that the quotient topology on $G\backslash \Prob(V)$ is $T_0$.
By Proposition~\ref{tame-orbits}, it is enough to check that the quotient topology on $G\backslash V$ is $T_0$,
which is guaranteed by the assumption that the action of $G$ on $V$ is almost algebraic, and, as we will see, that 
for any $v\in V$, the action of $G$ on $\Prob(Gv)$ is almost algebraic.
We note that by Effros theorem (Theorem~\ref{effros}), the orbit $Gv$ is equivariantly homeomorphic to the coset space $G/\Stab_G(v)$,
and thus $\Prob(Gv)\simeq \Prob(G/\Stab_G(v))$.
Since $\Stab_G(v)$ is an almost algebraic subgroup of $G$, the fact that the $G$-action on $\Prob(Gv)$ 
is almost algebraic now follows from Theorem~\ref{tameomog}.
\end{proof}

\section{On bounded subgroups}

In this section, we essentially retain the setup~\ref{setup} \& \ref{setupG}: we fix a complete $(k,|\cdot|)$ valued field and a $k$-algebraic group $\bfG$.
Nevertheless there is no need for us to assume that $(k,|\cdot|)$ is separable, so we will refrain from doing so.


\begin{definition}
A subset of $k$ is called bounded if its image under $|\cdot|$ is bounded in $\bbR$.
For a $k$-variety $\bf{V}$, a subset of ${\bf V}(k)$ is called bounded if its image by any regular function is bounded in $k$. 
\end{definition}


\begin{remark}
Note that the collection of bounded sets on a $k$-variety forms a bornology.
\end{remark}

\begin{remark}
For a $k$-variety $\bf{V}$ it is clear that a subset of ${\bf V}(k)$ is bounded if and only if its intersection with every $k$-affine open set is bounded,
so in what follows we will lose nothing by considering exclusively $k$-affine varieties.
We will do so.
\end{remark}

\begin{remark}
Note that if $(k',|\cdot|')$ is a field extension of $k$ endowed with an absolute value extension of $|\cdot|$
and ${\bf V}$ is a $k$-variety, we may regard ${\bf V}(k)$ as a subset ${\bf V}(k')$
and, as one easily checks, a subset of ${\bf V}(k)$ is $k$-bounded if and only if it is $k'$-bounded.
Thus it causes no loss of generality assuming $k$ is algebraically closed since $\widehat{k}$ is so. Nevertheless, we will not assume that.
\end{remark} 

It is clear that every $k$-regular morphism of $k$-varieties is a bounded
map in the sense that the image of a bounded set is bounded.
For a $k$-closed immersion of $k$-varieties $f:\bf{U}\to \bf{V}$ also the converse is true:
a subset of ${\bf U}(k)$ is bounded if and only if its image is bounded,
as $f^*:k[{\bf V}]\to k[{\bf U}]$ is surjective.
This is a special case of the following lemma.

\begin{lemma} \label{finite morphisms}
For a finite $k$-morphism $f:\bf{U}\to \bf{V}$
a subset of ${\bf U}(k)$ is bounded if and only if its image is bounded.
\end{lemma}

\begin{proof}
Assume there exists an unbounded set $L$ in ${\bf U}(k)$ with 
$f(L)$ being bounded in ${\bf V}(k)$.
Then we could find $p\in k[{\bf U}]$
and a sequence $u_n \in L$ with $|p(u_n)|\to\infty$.
The function $p$ is integral over $f^*k[{\bf V}]$
so there exist $q_1,\ldots q_m\in f^*k[{\bf V}]$
with $p^m+\sum_{i=1}^m q_ip^{m-i} = 0$.
Thus,
\[ 1 = \left|\sum_{i=1}^m \frac{q_i(u_n)}{p^i(u_n)}\right| \leq \sum_{i=1}^m \frac{|q_i(u_n)|}{|p^i(u_n)|} \to 0, \]
as the sequences $q_i(u_n)$ are uniformly bounded. This is a contradiction.
\end{proof}

Recall that a seminorm on a $k$-vector space $E$ is a function $\|\cdot\|:E\to [0,\infty)$ satisfying
\begin{enumerate}
\item $\|\alpha v\|=|\alpha|\|v\|$, for $\alpha\in k$, $v\in E$.
\item $\|u+v\|\leq \|u\|+\|v\|$, for $u,v\in E$.
\end{enumerate}

A seminorm on $E$ is a norm if furthermore we have
\begin{enumerate}[resume]
\item  $\|v\|=0~\Leftrightarrow~v=0$, for $v\in E$.
\end{enumerate}

Two norms on a vector space, $\|\cdot\|,\|\cdot\|'$, are called equivalent if there exists some $C \geq 1$ such that
\[ C^{-1}\|\cdot\| \leq \|\cdot\|' \leq C\|\cdot\|. \]

It is a general fact that any linear map between two Hausdorff topological $(k,|\cdot|)$-vector spaces of finite dimensions is continuous \cite[I, \S 2,3 Corollary 2]{BBK-EVT} and thus we get easily the following.

\begin{theorem}\label{thm:equiv}
All the norms on a finite dimensional $k$-vector space are equivalent.
\end{theorem}

\begin{proof}
It suffices to use that the identity map $(E,||\cdot||)\to(E,||\cdot||')$ is continuous and observe that every continuous linear map is bounded.
The latter is an easy exercise in case $|\cdot|$ is trivial, and standard if it is not.
\end{proof}

Recall that, if $(e_1,\dots,e_n)$ is a basis, then the norm $||\cdot||_\infty$ (relative to this basis) is defined as $||\sum x_i e_i||_\infty = \max\{ |x_i| \}$.

\begin{cor} \label{bounded conditions}
For a subset $B \subset E=k^n$ the following properties are equivalent.
\begin{enumerate}
\item $B$ is a bounded set of $\mathbb {A}^n$.
\item All elements of $E^*$ are bounded on $B$.
\item All the coordinates of the elements of $B$ are uniformly bounded.
\item The norm $||\cdot||_\infty$ is bounded on $B$.
\item Every norm on $E$ is bounded on $B$.
\item Some norm on $E$ is bounded on $B$.
\end{enumerate}
\end{cor}


%
%
%
%


\begin{theorem} \label{orthogonal}
For a subgroup $L$ of $\GL_n(k)$ the following are equivalent:
\begin{enumerate}
\item $L$ is bounded in $\GL_n(k)$.
\item $L$ is bounded as a subset of $\M_n(k)$.
\item $L$ preserves a norm on $k^n$.
\item $L$ preserves a spanning bounded set in $k^n$.
\end{enumerate}
For a subgroup $L$ of $G={\bf G}(k)$ the following are equivalent:
\begin{enumerate}
\item $L$ is bounded.
\item $L$ preserves a norm in all $k$-linear representations of ${\bf G}$.
\item $L$ preserves a norm in some
injective $k$-linear representation of ${\bf G}$.
\item $L$ preserves a spanning bounded set in some
injective $k$-linear representation of ${\bf G}$.
\end{enumerate}
\end{theorem}

\begin{proof}
Note that the second part of the theorem follows from the first once we recall that any injective homomorphism of algebraic groups is a closed immersion.
We prove the equivalence of the first four conditions.

$(1) \Leftrightarrow (2)\colon$
Clearly, if $L$ is bounded in $\GL_n(k)$ then it is bounded in $\M_n(k)$.
Assume $L$ is bounded in $\M_n(k)$. Then it has a bounded image under both morphisms
\[
\GL_n \xrightarrow{\iota} \GL_n
\hookrightarrow \M_n,
\quad
\GL_n \hookrightarrow \M_n \xrightarrow{\det} \mathbb{A}^1, \]
where
 \( \GL_n \xrightarrow{\iota} \GL_n \)
is the group inversion.
We conclude that $L$ has a bounded image under the product morphism
\( \GL_n \to \M_n\times \mathbb{A}^1\).
But the latter morphism is the composition of the isomorphism $\iota$ and the closed immersion
\( \GL_n \xrightarrow{\id\oplus\det^{-1}}
\M_n\oplus \mathbb{A}^1 \).
Thus $L$ is bounded in $\GL_n(k)$.

$(2) \Leftrightarrow (3)\colon$
If $L$ is bounded in $\M_n(k)$ then,
by
Corollary~\ref{bounded conditions}(3)
all its matrix elements are uniformly bounded,
hence for all $v\in k^n$, $\sup_{g\in L} \|gv\|_\infty$ is finite.
This expression forms an $L$-invariant norm on $k^n$.
On the other hand, if $L$ preserves a norm on $k^n$,
by the equivalence of this norm with $\|\cdot\|_\infty$, all matrix elements of $L$ are uniformly bounded, thus it is bounded in $\M_n(k)$.

$(3) \Leftrightarrow (4)\colon$
If $L$ preserves a norm then it preserves its unit ball which is a bounded spanning set.
If $L$ preserves a bounded spanning set $B$ than it also preserves its symmetric convex hull:
\[ \left\{\left.\sum_{i=1}^n \alpha_iv_i\right| v_i\in B,~\alpha_i\in k,~\sum_{i=1}^n |\alpha_i|\leq 1\right\}. \]
The latter is easily seen to be the unit ball of an $L$-invariant norm.
\end{proof}

Note that if $L$ is a compact subgroup of $G$ then $L$ is bounded, as the $k$-regular functions of ${\bf G}$ are continuous on $G$.

\begin{cor} \label{bounded-bi-inv}
Every bounded subgroup of $G$ admits a bi-invariant metric.
\end{cor}

\begin{proof}
Let $L$ be a bounded subgroup of $G$.
Fix an injective $k$-linear representation ${\bf G}\to \GL(V)$ and consider $L$ as a subset of $\End_k(V)$.
$\End_k(V)$ is a representation of $G\times G$, hence admits a norm which is invariant under the bounded group $L\times L$ by Theorem~\ref{orthogonal}.
This norm gives an $L\times L$-invariant metric on $\End_k(V)$ and on its subset $L$.
\end{proof}

\begin{prop}
Assume $\bf{V}$ is an affine $k$-variety with a $k$-affine action of $G$.
Let $B\subset {\bf V}(k)$ be a bounded set and denote by $\overline{B}^Z$ its Zariski closure.
Then the image of $\Stab_G(B)$ is bounded in the $k$-algebraic group  $\Stab_{\bf{G}}(\overline{B}^Z)/\Fix_{\bf{G}}(\overline{B}^Z)$.
\end{prop}

\begin{proof}
Without loss of generality we may replace $\bf{G}$ by $\Stab_{\bf{G}}(\overline{B}^Z)$ and then assume ${\bf{V}}=\overline{B}^Z$.
We then may further assume ${\bf{G}}=\Stab_{\bf{G}}(\overline{B}^Z)/\Fix_{\bf{G}}(\overline{B}^Z)$. We do so.
By \cite[Proposition~1.12]{borel} there exists a $k$-embedding of ${\bf V}$ into some vector space, which we may assume having a spanning image,
equivariant with respect to some $k$-representation ${\bf G}\to\GL_n$, which we thus may assume injective.
The proof then follows from the implication $(4)\implies(1)$ in the second of equivalence of Theorem~\ref{orthogonal}.
\end{proof}

\begin{cor} \label{NL}
Let $L<G$ be a bounded subgroup.
Then $N_{\bf G}(L)/Z_{\bf G}(L)$ is bounded.
\end{cor}

\section{The space of norms and seminorms}\label{norms}

In this section we study a compact space on which an algebraic group over a complete valued field acts by homeomorphisms,
the space of seminorms. This space was already considered in the case when $k$ is local, in \cite{werner}.


%
%
%
We fix a finite dimensional vector space $E$ over $k$. 
Given two norms $n,n'$ on $E$ we denote 
\[ d(n,n')=\log\sup\left\{\left.\frac{n(y)n'(x)}{n'(y)n(x)}~\right|~x,y\in E\setminus\{0\}\right\}. \]
This number is finite by the fact that $n$ and $n'$ are equivalent norms.
Recall that two seminorms on $E$ are called homothetic if they differ by a multiplicative positive constant.
The relation of being homothetic is an equivalence relation.
We denote the set consisting of all homothety classes of norms on $E$ by $I(E)$.
Observe that $d(n,n')$ only depends on the homothety classes of $n$ and $n'$ and thus define a function on $I(E)$.


\begin{lemma}\label{boundedaction}
The function $d:I(E)\times I(E)\to [0,\infty)$ defines  a metric on $I(E)$. The group $\PGL(E)$ acts continuously and isometrically on $I(E)$ and the stabilizers in $\PGL(E)$ of bounded subsets in $I(E)$ are bounded as well.
\end{lemma}

\begin{proof}
The fact that $d$ is a metric and $\PGL(E)$ acts by isometries on $I(E)$ is a straightforward verification. 
To prove the continuity part, it suffices to show that the orbit map $g\mapsto gn$ is continuous for all $n\in I(E)$. 
Fix a norm $n$ on $E$. 
Let $(g_i)$ be a sequence converging to $e$ in $\PGL(E)$. 
By an abuse of notation we identify $g_i$ with an element of $\GL(E)$ such that $g_i\to e\in\GL(E)$, and also still denote $n$ a norm whose homothety class is $n$.
Using that $d(g_in,n)=\log\sup\left\{\frac{n(g_i^{-1}y)n(x)}{n(g_i^{-1}x)n(y)}~|~\ x,y\in E\setminus \{0\},\ n(x),n(y)<1\right\}$ and that $g_i^{-1}$ converges uniformly to $e$ on the unit ball of $E$ with respect to $n$, we see that indeed $d(g_in,n)\to 0$.

Let $L$ be the stabilizer of some bounded set $N\subseteq I(E)$. 
Fix $v\neq 0$ and identify $N$ with a set $N'$ of norms on $E$ satisfying $n(v)=1$ for every $n\in N'$. 
The set $B=\{x\in E~|~\forall n\in N',~n(x)\leq 1\}$ is clearly bounded in $E$.
By Theorem~\ref{orthogonal}, its stabilizer $L'\in \GL(E)$ is bounded, hence also its image in $\PGL(E)$, namely $L$.
\end{proof}

\begin{remark} The space $I(E)$  actually contains the affine Bruhat-Tits building $\mathcal{I}(E)$ associated to $\PGL(E)$ \cite{Parreau} and there is a metric $d_0$ on $\mathcal{I}(E)$ such that $(\mathcal{I}(E),d_0)$ is CAT(0) ---not necessarily complete. The metric $d$ is similar to the one considered by Goldman and Iwahori in \cite{MR0144889}. The two metrics $d$ and $d_0$ are Lipschitz-equivalent. 
This can be checked first on an apartment and extended to the whole building using that any two points actually lie in some apartment. Thus, Lemma \ref{boundedaction} and Theorem~\ref{orthogonal} are a reminiscence of the Bruhat-Tits fixed point theorem.
\end{remark}


Let  $S'(E)$ be the space of non-zero seminorms on $E$, and $S(E)$ be its quotient by homotheties.
 We endow $S'(E)$ with the topology of pointwise convergence and $S(E)$ with the quotient topology.

\begin{prop}
The space $S(E)$ is compact and metrizable. The action of $\PGL(E)$ on $S(E)$ is continuous.
\end{prop}

\begin{proof} 
 Let $m$ be the dimension of $E$. Fix a basis $(e_1,\dots, e_m)$ of $E$. Let $S_1(E)$ be the set of all $s\in S'(E)$ such that $s(e_i)\leq 1$ for every $1\leq i\leq d$, and $s(e_j)=1$ for some $j$.

We first claim that the quotient map $S'(E)\to S(E)$ restricts to a surjection $S_1(E)\to S(E)$. This follows from the fact that if a seminorm is zero on all the vectors $e_i$, then it is zero everywhere, by triangle inequality. Furthermore, the map $S_1(E)\to S(E)$ is actually an injection. Indeed if $s\in S_1(E)$ and $\lambda s\in S_1(E)$ it is easy to conclude that $\lambda=1$.

We now claim that the space $S_1(E)$ is compact. Indeed, let $\|\cdot\|_1$ be the norm defined as $\|\sum x_i e_i \|_1=\sum_i |x_i | $. Let $v=\sum x_ie_i$. Then  we see that $s(v)\leq  \|v\|_1$ for every $v\in E$.  So we get that $S_1(E)$ is homeomorphic to a closed subset of $\prod_{v\in E} \left[0,\|v\|_1\right]$. This proves the compactness of $S_1(E)$ and therefore of $S(E)$.  Note that it also proves that every element of $S_1(E)$ is $1$-Lipschitz with respect to $\|\cdot \|_1$.

It follows that $S_1(E)$ is homeomorphic to $S(E)$. The metrizability of $S_1(E)$  comes from the fact that $S_1(E)$ is a closed subset of the space of continuous functions on $E$, which is metrizable because $E$ is separable. 

Now,  let $(g_n,s_n)$ be a sequence converging to $(e,s)\in \GL(E)\times S_1(E)$ then $g_ns_n$ tends to $s$. Indeed, for every $v\in E$,
\[|s_n(g_nv)-s(v)|\leq|s_n(g_nv)-s_n(v)|+|s_n(v)-s(v)|\leq \|g_nv-v\|_1+| s_n(v)-s(v)|\to0.\] 
\end{proof}

Each non-zero seminorm $s$ has a kernel $\ker(s)=\{v\in E ~|~ s(v)=0\}$, 
which is a proper linear subspace of $E$ depending only of the homothety class of $s$. 
The map $S(E)\to \bbN$, $s\mapsto \dim(\ker(s))$ is obviously $\PGL(E)$-invariant.
Denote by $S_m(E)$ the space of homothety classes of seminorms $s$ such that $\dim(\ker(s))=m$. Note that $S_0(E)=I(E)$.
We denote by $\Gr_m(E)$ the Grassmannian of $m$-dimensional linear subspaces of $E$.
The map  $S_m(E)\to \Gr_m(E)$, $s\mapsto \ker(s)$ is clearly $\PGL(E)$-equivariant.
$\Gr_m(E)$ is the $k$-points of a $k$-algebraic variety, thus carries a Polish topology by Proposition~\ref{polishing}.


\begin{prop}\label{ker}
The maps $S(E)\to \bbN$, $s\mapsto \dim(\ker(s))$ and $S_m(E)\to \Gr_m(E)$, $s\mapsto \ker(s)$ are measurable.
\end{prop}

\begin{proof}
We first note that the space $S(E)$ is covered by (countably many) open sets which are homeomorphic images of sets of the form 
$\{s\in S'(E)~|~s(v)=1\}$, for $v\in E$, under the quotient map $S'(E)\to S(E)$.
It is therefore enough to establish that the corresponding maps $S'(E)\to \bbN$, $S_m'(E)\to \Gr_m(E)$ are measurable (where $S'_m(E)$ denotes the preimage of $S_m(E)$).

Fix a basis for $E$ and a countable dense subfield $k_0<k$.
Let $E_0=E(k_0)$ be the $k_0$-span of the fixed basis of $E$.
A subspace of $E$ is said to be defined over $k_0$ if it has a basis in $E_0$.
$E_0$ is a $k_0$-vector space and it is a countable dense subset of $E$.
Note that for every $d$, $\Gr_m(E_0)$ is countable.
Observe that for $s\in S'(E)$, $\dim(\ker(s))\leq m$ if and only if we can find a codimension $m$ subspace $F<E$ which is defined over $k_0$, such that $s$ restricts to a norm on $F$.
The latter condition is equivalent by Theorem~\ref{thm:equiv} to the condition that there exists $n\in \bbN$ such that for every $v\in F$, $s(v)\geq |v|/n$ for some fixed norm $|\cdot|$.
Note that it is enough to check this for every $v\in F_0=F(k_0)$, thus we obtain
\[ \{s\in S'(E)~|~\dim(\ker(s))\leq m\} = \bigcup_{F_0\in \Gr_{\dim(E)-m}(E_0)} \bigcup_{n} \bigcap_{v\in F_0} \{s\in S'(E)~|~s(v)\geq |v|/n\}. \]
This shows that the map $s\mapsto \dim(\ker(s))$ is measurable.

In order to prove that the map $S'_m(E) \to \Gr_m(E)$ is measurable, we make two observations.
We first observe that the topologies of pointwise convergence and uniform convergence give the same Borel structure on $S'(E)$.
In fact, for every separable topological space $X$, the pointwise and uniform convergence topologies on $C_b(X)$ give the same Borel structure
(as uniform balls are easily seen to be Borel for the pointwise convergence topology), and $S'(E)$ could be identified with a closed (for both topologies) subspace of bounded continuous functions on the unit ball of $E$.
Our second observation is that we may identify $\Gr_m(E)$ with a subset of the space of closed subsets of the unit ball of $E$.
Endowing it with the Hausdorff metric topology, we get a $\PGL(E)$-invariant Polish topology on $\Gr_m(E)$.
Since the Polish group $\PGL(E)$ acts transitively on $\Gr_m(E)$, by 
Effros Lemma \cite[Lemma 2.5]{effros} the quotient topology is the unique $\PGL(E)$-invariant Polish topology on this space,
thus the topology on $\Gr_m(E)$ given by the Hausdorff metric coincides with the one discussed in Proposition~\ref{polishing}.

The proof is now complete, observing further that with respect to the uniform convergence topology on $S'_m(E)$ and the Hausdorff metric topology on $\Gr_m(E)$, the map $s\mapsto \ker(s)$ is in fact continuous
(moreover, it is $C$-Lipschitz on $\{s\in S'_m(E)~|~s~\textrm{ is }C\mbox{-Lipschitz}\}$).
\end{proof}

\section{Existence of algebraic representations}

This section is devoted to the proof of Theorem~\ref{BDL},
which we restate below. The reader who is unfamiliar with the notion of measurable cocycles and amenable actions might consult with profit Zimmer's book \cite[Chapter~4]{zimmer-book}. The following theorem provides a so-called \emph{algebraic representation} of the space $R$, thus allowing to start the machinery developed in \cite{BF} and prove cocycle super-rigidity for the group $G$.


\begin{theorem}
Let $R$ be a locally compact group and $Y$ an ergodic, amenable Lebesgue $R$-space.
Let $(k,|\cdot|)$ be a valued field.
Assume that as a metric space $k$ is complete and separable.
Let ${\bf G}$ be a simple $k$-algebraic group.
Let $f:R\times Y \to {\bf G}(k)$ be a measurable cocycle.

Then either
there exists
a $k$-algebraic subgroup ${\bf H}\lneq {\bf G}$
and an $f$-equivariant measurable map $\phi:Y\to {\bf G}/{\bf H}(k)$,
or there exists
a complete and separable metric space $V$ on which $G$ acts by isometries
with bounded stabilizers
and an $f$-equivariant measurable map $\phi':Y\to V$.

Furthermore,
in case $k$ is a local field the $G$-action on $V$ is proper
and in case $k=\bbR$ and $G$ is non-compact the first alternative always occurs.
\end{theorem}




\begin{proof}
We first note that the isogeny ${\bf G}\to \overline{\bf G}$, where $\overline{\bf G}$ is the adjoint group associated to ${\bf G}$, is a finite morphism.
Thus, by Lemma~\ref{finite morphisms} we may assume that ${\bf G}$ is an adjoint group. 
We do so.
By \cite[Proposition~1.10]{borel} we can find a $k$-closed immersion from ${\bf G}$ into some $\GL_n$. 
By the fact that ${\bf G}$ is simple, we may assume that this representation is irreducible.
By the fact that ${\bf G}$ is adjoint, the associated morphism ${\bf G}\to \PGL_n$ is a closed immersion as well.
We will denote for convenience $E=k^n$.
Via this representation, $G$  acts continuously and faithfully on the metric space of homothety classes of norms, $I(E)$,
and on the compact space of homothety classes of seminorms, $S(E)$, introduced in \S\ref{norms}.

By the amenability of the action of $R$ on $Y$ there exists a $f$-map, that is a $f$-equivariant map, $\phi\colon Y\to\Prob(S(E))$, which we now fix.
By Proposition~\ref{ker}, there is a measurable partition $S(E)=\cup_{d=0}^{n-1} S_d(E)$, 
given by the dimension of the kernels of the seminorms.
For a given $d$, the function $Y\to [0,1]$ given by $y\mapsto \phi(y)(S_d(E))$ is $R$-invariant, hence almost everywhere equal to some constant, by ergodicity.
We denote this constant by $\alpha_d$.
Note that $\sum_{d=0}^{n-1} \alpha_d =1$. We choose $d$ such that $\alpha_d>0$ and define 
\[ \psi:Y\to \Prob(S_d(E)), \quad \psi(y)=\frac{1}{\alpha_d}\phi(y)|_{S_d(E)}. \]
Note that $\psi$ is a $f$-map.
We will consider two cases: either $d>0$ or $d=0$. This is a first bifurcation leading to the two alternatives in the statement of the theorem.

We first consider the case $d>0$.
We use the map $S_d(E)\to \Gr_d(E)$ discussed in Proposition~\ref{ker} to obtain the push forward map $\Prob(S_d(E))\to \Prob(\Gr_d(E))$.
By post-composition, we obtain a $f$-map $\Psi:Y\to \Prob(\Gr_d(E))$.
By Theorem~\ref{mainthm} the action of $G$ on $\Prob(\Gr_d(E))$ is almost algebraic (as the action of $G$ on $\Gr_d(E)$ is almost algebraic by Proposition~\ref{polishing}),
and the quotient topology on $G\backslash\Prob(\Gr_d(E))$ is $T_0$.
We claim that there exists $\mu \in \Prob(\Gr_d(E))$ such that the set $\Psi^{-1}(G\mu)$ has full measure in $Y$.
The standard argument is similar to the prof of Proposition~\ref{tame}:
for a countable basis $B_i$ for the topology of $G\backslash \Prob(\Gr_d(E))$,
the set 
\[ \bigcap \{B_i ~|~\Psi^{-1}(B_i) \mbox{ is full in }Y\}\cap \bigcap\{B_i^c ~|~\Psi^{-1}(B_i)\mbox{ is null in }Y\} \]
is clearly a singleton, whose preimage is of full measure in $Y$. 
Let $\mu$ be a preimage of this singleton in $\Prob(\Gr_d(E))$.

By the fact that $G$ acts almost algebraically on $\Prob(\Gr_d(E))$, we may identify $G\mu$ with a coset space $G/L$, for some almost algebraic subgroup $L=\Stab_G(\mu)<G$,
and view $\Psi$ as an $f$-map from $Y$ to $G/L$.
By Proposition~\ref{prop:stab}, there exists a
$k$-subgroup ${\bf H}_0<{\bf G}$ which is normalized by $L$
such that $L$ has a precompact image in the Polish group $(N_{\bf G}({\bf H}_0)/{\bf H}_0)(k)$
and such that $\mu$ is supported on the subvariety of ${\bf H}_0$ fixed points in $\Gr_d(E)$.
Note that by the irreducibility of the representation ${\bf G} \to \GL_n$ we have no ${\bf G}$-fixed points in $\Gr_d(E)$,
thus ${\bf H}_0 \lneq {\bf G}$.

Assume moreover that ${\bf H}_0\neq \{e\}$ and let ${\bf H}$ be the Zariski-closure of $L$.
By \cite[Theorem~AG14.4]{borel}, ${\bf H}$ is a $k$-subgroup of ${\bf G}$.
By the simplicity of ${\bf G}$, ${\bf H}\lneq {\bf G}$, as ${\bf H}$ normalizes ${\bf H}_0$.
Post-composing the $f$-map $\Psi$ with the map $G/L \to {\bf G}/{\bf H}(k)$ we obtain
a $k$-algebraic subgroup ${\bf H}\lneq {\bf G}$
and an $f$-equivariant measurable map $\phi:Y\to {\bf G}/{\bf H}(k)$, as desired.

Assume now ${\bf H}_0=\{e\}$. 
In that case $L$ is compact, and in particular bounded in $G$.
It follows by Theorem~\ref{orthogonal} that $L$ fixes a norm on $E$.
Thus we may map the coset space $G/L$ $G$-equivariantly into $S_0(E)=I(E)$.
Using the $\delta$-measure embedding $I(E)\hookrightarrow\Prob(I(E))$
and obtain a new $f$-map $Y\to \Prob(I(E))$. We are then reduced to the case $d=0$, to be discussed below.

We consider now the case $d=0$, that is we assume having an $f$-map $Y\to \Prob(I(E))$.
We set $V=\Prob(I(E))$.
By Lemma~\ref{boundedaction}, $G$ acts isometrically and with bounded stabilizers on $I(E)$.
By Lemma~\ref{contprob}, $G$ acts isometrically on $V$.
Let us check that stabilizers are bounded. 
Fix $\mu\in \Prob(I(E))$, and let $L$ be its stabilizer in $G$. 
Since $I(E)$ is Polish there is a ball $B$ of $I(E)$ such that $\mu(B)>1/2$.
It follows that for any $g\in L$, $gB$ intersects $B$.
Thus the set $LB$ is bounded in $I(E)$, and by Lemma \ref{boundedaction} its stabilizer is bounded in $G$. 
It follows that $L$ is bounded.
Thus we have found an $f$-map from $Y$ to a complete and separable metric space $V$ on which $G$ acts by isometries
with bounded stabilizers as desired.
\end{proof}

\section*{Acknowledgement}
U.B was supported in part by the ERC grant 306706. B.D. is supported in part by Lorraine Region and Lorraine University. B.D \& J.L. are supported in part by ANR grant ANR-14-CE25-0004 GAMME.

\bibliographystyle{smfalpha}
\bibliography{biblio}

\end{document}